\documentclass[a4paper,reqno]{amsart}

\usepackage{amssymb,upref}
\usepackage[mathcal]{euscript}
\usepackage[cmtip,all]{xy}

\usepackage{tikz}

\tikzstyle cross=[preaction={draw=white,-,line width=4pt},thick]

\tikzstyle bolito=[draw,circle,fill=black,minimum size=3pt,inner sep=0pt]
                            
\tikzstyle punto=[draw,circle,fill=black,minimum size=1pt,inner sep=0pt]
                            
\tikzstyle fino=[line width=0.5pt]

\newcommand{\tik}[2]{%
 \begin{tikzpicture}[scale=#1,baseline={([yshift=-3pt]current bounding box.center)},label distance=-3pt] #2 \end{tikzpicture}}

\newcommand{\bydef}{:=}

\newcommand{\vphi}{\varphi}
\newcommand{\veps}{\varepsilon}

\newcommand{\wb}[1]{\overline{#1}}

\newcommand{\id}{\mathrm{id}}

\newcommand{\espan}[1]{\mathrm{span}\left\{#1\right\}}

\newcommand{\diag}{\mathrm{diag}}

\newcommand{\cA}{\mathcal{A}}
\newcommand{\cB}{\mathcal{B}}
\newcommand{\cC}{\mathcal{C}}

\newcommand{\cH}{\mathcal{H}}

\newcommand{\cK}{\mathcal{K}}
\newcommand{\cL}{\mathcal{L}}

\newcommand{\cO}{\mathcal{O}}
\newcommand{\cQ}{\mathcal{Q}}
\newcommand{\cR}{\mathcal{R}}
\newcommand{\cS}{\mathcal{S}}

\newcommand{\cU}{\mathcal{U}}
\newcommand{\cV}{\mathcal{V}}
\newcommand{\cW}{\mathcal{W}}

\newcommand{\frf}{{\mathfrak f}}

\newcommand{\frs}{{\mathfrak s}}

\newcommand{\ZZ}{\mathbb{Z}}

\newcommand{\RR}{\mathbb{R}}
\newcommand{\CC}{\mathbb{C}}
\newcommand{\HH}{\mathbb{H}}

\newcommand{\FF}{\mathbb{F}}

\DeclareMathOperator{\charac}{\mathrm{char}}

\DeclareMathOperator{\AlgF}{\mathrm{Alg_{\FF}}}
\DeclareMathOperator{\rad}{\mathrm{rad}}
\DeclareMathOperator{\Hom}{\mathrm{Hom}}

\DeclareMathOperator{\Aut}{\mathsf{Aut}}
\DeclareMathOperator{\SStab}{\mathbf{Stab}}
\DeclareMathOperator{\AAut}{\mathbf{Aut}}
\DeclareMathOperator{\Der}{\mathrm{Der}}

\DeclareMathOperator{\alg}{\mathrm{alg}}

\DeclareMathOperator{\Fix}{\mathrm{Fix}}

\DeclareMathOperator{\Centr}{\mathrm{Cent}}
\DeclareMathOperator{\CCentr}{\mathbf{Cent}}
\DeclareMathOperator{\SSU}{\mathbf{SU}}
\DeclareMathOperator{\SSL}{\mathbf{SL}}

\newcommand{\bfT}{\mathbf{T}}
\newcommand{\bfG}{\mathbf{G}}
\newcommand{\bfU}{\mathbf{U}}

\newcommand{\bfP}{\mathbf{P}}
\newcommand{\bfL}{\mathbf{L}}
\newcommand{\bfB}{\mathbf{B}}
\newcommand{\bfH}{\mathbf{H}}
\newcommand{\bfK}{\mathbf{K}}
\newcommand{\bfN}{\mathbf{N}}

\newcommand{\ad}{\mathrm{ad}}
\newcommand{\Ad}{\mathrm{Ad}}

\newcommand{\frsl}{{\mathfrak{sl}}}

\newcommand{\frpsl}{{\mathfrak{psl}}}

\newcommand{\sfS}{\mathsf{S}}
\newcommand{\sfC}{\mathsf{C}}

\def\hregleta{\hrule height .5pt}
\def\hreglon{\hrule height1pt}
\def\vreglon{\vrule height 12pt width1pt depth 4pt}
\def\vregleta{\vrule width .5pt}

\def\hregletafill{\leaders\hregleta\hfill}

\newcommand{\subo}{_{\bar 0}}
\newcommand{\subuno}{_{\bar 1}}

\newcommand\boldmu{\boldsymbol{\mu}}

\renewcommand{\Hom}{\mathsf{Hom}}

\renewcommand{\id}{\mathsf{id}}

\newcommand{\cl}{\mathsf{cl}}

\newtheorem{theorem}{Theorem}
\newtheorem{proposition}[theorem]{Proposition}
\newtheorem{lemma}[theorem]{Lemma}
\newtheorem{corollary}[theorem]{Corollary}

\theoremstyle{definition}
\newtheorem{df}[theorem]{Definition}
\newtheorem{example}[theorem]{Example}
\newtheorem{remark}[theorem]{Remark}

\numberwithin{theorem}{section}
\numberwithin{equation}{section}

\newenvironment{romanenumerate}
 {\begin{enumerate}
 
 }{\end{enumerate}}

\begin{document}

\title[Order $3$ elements in $G_2$ and idempotents]{Order $3$ 
elements in $G_2$ and idempotents in symmetric composition algebras}

\author[A.~Elduque]{Alberto Elduque${}^\star$}
\address{Departamento de Matem\'{a}ticas
 e Instituto Universitario de Matem\'aticas y Aplicaciones,
 Universidad de Zaragoza, 50009 Zaragoza, Spain}
\email{elduque@unizar.es}
\thanks{${}^\star$ Supported by the Spanish Ministerio de Econom\'{\i}a y 
Competitividad---Fondo Europeo de Desarrollo Regional (FEDER) MTM2013-45588-C3-2-P
}




\date{}

\begin{abstract}
Order three elements in the exceptional groups of type $G_2$ are classified up to conjugation over arbitrary fields. Their centralizers are computed, and the associated classification of idempotents in symmetric composition algebras is obtained. Idempotents have played a key role in the study and classification of these algebras.

Over an algebraically closed field, there are two conjugacy classes of order three elements in $G_2$ in characteristic $\neq 3$ and four of them in characteristic $3$. The centralizers in characteristic $3$ fail to be smooth for one of these classes.
\end{abstract}

\maketitle



\section{Introduction}\label{se:intro}

Symmetric composition algebras constitute an important class of algebras strongly related to the triality phenomenon (see \cite[Chapter VIII]{KMRT}). They were classified in \cite{EM91,EM93} over fields of characteristic $\ne 3$ and in \cite{Eld97} in characteristic $3$.

Idempotents in these algebras have proved to be a key tool in this classification, and these idempotents, in dimension $8$, are related to order $3$ automorphisms of Cayley algebras, that is, to order $3$ rational elements in simple linear algebraic groups of type $G_2$: the groups of automorphisms of Cayley algebras.

The goal of this paper is to classify, up to conjugation, the order $3$ elements in the groups of automorphisms of Cayley algebras over arbitrary fields, and to deduce from here a complete description of the idempotents in symmetric composition algebras.

Over an algebraically closed field of characteristic $\neq 3$, it is easy to check that, up to conjugation, there are two different order $3$ elements, which correspond to the two nonisomorphic eight-dimensional symmetric composition algebras: the para-Cayley algebra and the Okubo algebra. Over arbitrary fields of characteristic $\neq 3$, still there are two types of order $3$ elements in the group of automorphisms of a Cayley algebra (Theorem \ref{th:char_not3}).

The situation is more interesting in characteristic $3$, where the existence of a unique `quaternionic idempotent' in the split Okubo algebra was of crucial importance in the determination of the affine group scheme of automorphisms of this algebra in \cite{CEKT}, and explains why the group of automorphisms is a subgroup of the exceptional algebraic group of type $G_2$. In characteristic $3$, once a torus is fixed on the linear algebraic group of type $G_2$ over an algebraically closed field, for any root $\alpha$ relative to this torus and any nonzero scalar $t$, the element ``$\exp(tx_{\alpha})$'' has order $3$. It turns out that there are four different types of order $3$ elements in the automorphism group of a split Cayley algebra over an arbitrary field of characteritic $3$ (Theorem \ref{th:char3}).

In Section \ref{se:composition_algebras}, the basic definitions and results on both unital composition algebras and on symmetric composition algebras will be recalled. The construction of the split Cayley algebra as the algebra of Zorn matrices will be reviewed. In Section \ref{se:Zorn_autos}, a specific Chevalley basis of the Lie algebra of type $G_2$ will be given, and some noteworthy subgroups of automorphisms will be considered.

Section \ref{se:paraCayley} will highlight some specific order $3$ automorphisms of Cayley algebras related to para-Cayley algebras. Idempotents in para-Cayley algebras and in symmetric composition algebras of dimension $\neq 8$ will be described too in this section.

Section \ref{se:char_not_3} is devoted to describe the conjugacy classes of order $3$ elements in the group of automorphisms of a Cayley algebra over a field of characteristic $\neq 3$, together with their centralizers, and this is used to describe the idempotents in eight-dimensional symmetric composition algebras over these fields. The split Okubo algebra is characterized as the only Okubo algebra with isotropic norm and idempotents (Theorem \ref{th:O_split_e}). The number of conjugacy classes of idempotents depends on the ground field.

The remaining sections will deal with the much more difficult case of  fields of characteristic $3$. Section \ref{se:char3_autos} will classify the conjugacy classes of order $3$ automorphisms of Cayley algebras. There are four different types of such automorphisms. Section \ref{se:centralizers_3} will be devoted to compute their centralizers. Some of them are not smooth. Finally, Section \ref{se:char3_idempotents} deals with idempotents of Okubo algebras over fields of characteristic $3$. In the split case, there is a unique \emph{quaternionic} idempotent, some conjugacy classes of \emph{quadratic} idempotents, the number of which depends on the ground field, and a unique conjugacy class of \emph{singular} idempotents (Theorem \ref{th:Idempotents_Okubo_3}). Over non perfect fields of characteristic $3$, there are non split Okubo algebras with idempotents, and all these idempotents are quadratic.

\smallskip

Throughout the paper, the functorial approach to algebraic groups (i.e.; algebraic affine group schemes) will be followed (see, e.g., \cite[Chapter VI]{KMRT}). Thus, an algebraic group is a representable functor from the category of unital, commutative, associative $\FF$-algebras $\AlgF$ to the category of groups, so that the (Hopf) algebra that represents it is finitely generated (as an algebra).

Any linear algebraic group $\bfG$ defined over a field $\FF$ in the classical sense (as in, e.g., \cite{Springer}) defines a smooth algebraic group which will be denoted too by $\bfG$. 

Also, for a finite dimensional algebra $(\cA,\cdot)$ over $\FF$, $\Aut(\cA,\cdot)$ will denote its automorphism group. This is the group of rational points of the algebraic group $\AAut(\cA,\cdot)$: $\Aut(\cA,\cdot)=\AAut(\cA,\cdot)(\FF)$. The Lie algebra of $\AAut(\cA,\cdot)$ is the Lie algebra of derivations $\Der(\cA,\cdot)$.

Note that, in general, $\AAut(\cA,\cdot)$ may fail to be smooth. This is what happens for the Okubo algebras over fields of characteristic $3$ (\cite[\S 10,11]{CEKT}).

\bigskip


\section{Composition algebras}\label{se:composition_algebras}

\begin{df}\label{df:composition}
A \emph{composition algebra} over a field $\FF$ is a triple $(\cC,\ast,n)$ such that
\begin{itemize}
\item $(\cC,\ast)$ is a nonassociative (i.e., not necessarily associative) algebra.
\item  $n:\cC\rightarrow\FF$ is a nonsingular multiplicative (i.e. $n(x\ast y)=n(x)n(y)$ for any $x,y\in\cC$) quadratic form.
\end{itemize}
The unital composition algebras are termed \emph{Hurwitz algebras}.
\end{df}

The quadratic form being nonsingular means that either $\{x\in\cC: n(x,\cC)=0\}$ is trivial, or it is a nonisotropic one-dimensional subspace. Here we will denote the polar form: $n(x,y)\bydef n(x+y)-n(x)-n(y)$, by the same symbol $n$.

The reader is referred to \cite[Chapter VIII]{KMRT} or to \cite[Chapter 2]{ZSSS} for the basic facts about composition algebras.

If the multiplication and the norm are clear from the context, we will simply refer to the algebra $\cC$.

Hurwitz algebras form a well-known class of composition algebras. Any element of a Hurwitz algebra $(\cC,\cdot,n)$ satisfies the `Cayley-Hamilton equation'
\[
x^{\cdot 2}-n(x,1)x+n(x)1=0.
\]
 Besides, $(\cC,\cdot,n)$ is endowed with an involution: the \emph{standard conjugation}, defined by
 \[
 \bar x\bydef n(x,1)1-x,
\]
which satisfies
\[
\bar{\bar x}=x,\quad \overline{x\cdot y}=\bar y\cdot\bar x,\quad x+\bar x=n(x,1)1,\quad x\cdot\bar x=\bar x\cdot x=n(x)1,
\]
for any $x,y\in\cC$.

The dimension of any Hurwitz algebra is finite and restricted to $1$, $2$, $4$ or $8$, and any Hurwitz algebra is either the ground field, a quadratic \'etale algebra, a quaternion algebra or a Cayley (or octonion) algebra (see \cite[(33.17)]{KMRT}).

\begin{example}\label{ex:Zorn}
The  \emph{Zorn matrix algebra}  is the compostion algebra $(\cC,\cdot,n)$ with 
\[
\cC=\left\{\begin{pmatrix} \alpha&u\\ v&\beta\end{pmatrix} : \alpha,\beta\in \FF, u,v\in\FF^3\right\},
\]
with 
\[
\begin{pmatrix} \alpha&u\\ v&\beta\end{pmatrix}\cdot\begin{pmatrix} \alpha'&u'\\ v'&\beta'\end{pmatrix}=
\begin{pmatrix} \alpha\alpha'+(u\mid v')&\alpha u'+\beta' u- v\times v'\\ 
\alpha' v+\beta v' + u\times u'&\beta\beta'+ (v\mid u')\end{pmatrix},
\]
and 
\[
n\left(\begin{pmatrix} \alpha&u\\ v&\beta\end{pmatrix}\right)=\alpha\beta -(u\mid v),
\]
where, for $u=(\mu_1,\mu_2,\mu_3)$ and $v=(\nu_1,\nu_2,\nu_3)\in\FF^3$, $(u\mid v)$ denotes the usual scalar product: $(u\mid v)=\sum_{i=1}^3\mu_i\nu_i$, and $u\times v$ the usual vector product: $u\times v=(\mu_2\nu_3-\mu_3\nu_2, \mu_3\nu_1-\mu_1\nu_3, \mu_1\nu_2-\mu_2\nu_1)$.

Let $\{a_1,a_2,a_3\}$ be the standard basis of $F^3$.
The elements
\[
 e_1 = \begin{pmatrix}
 1& 0\\[0.1ex] 0 &0 \end{pmatrix},\  f_1 = \begin{pmatrix}
 0& 0\\[0.1ex] 0 &1 \end{pmatrix},\
 u_i = \begin{pmatrix} 0 & -a_i\\[0.1ex] 0 &0 \end{pmatrix},\
  v_i = \begin{pmatrix} 0 & 0\\[0.1ex] a_i &0 \end{pmatrix},\ i=1,2,3,\
\]
form a hyperbolic basis, called
the \emph{canonical basis} of the Zorn algebra $(C,\cdot,n)$. The multiplication table in this basis is given in Table \ref{ta:good_basis}.
\begin{table}[!h]
\[
\vcenter{\offinterlineskip
\halign{\hfil$#$\enspace\hfil&#\vreglon
 &\hfil\enspace$#$\enspace\hfil
 &\hfil\enspace$#$\enspace\hfil&#\vregleta
 &\hfil\enspace$#$\enspace\hfil
 &\hfil\enspace$#$\enspace\hfil
 &\hfil\enspace$#$\enspace\hfil&#\vregleta
 &\hfil\enspace$#$\enspace\hfil
 &\hfil\enspace$#$\enspace\hfil
 &\hfil\enspace$#$\enspace\hfil&#\vreglon\cr
 &\omit\hfil\vrule width 1pt depth 4pt height 10pt
   &e_1&e_2&\omit&u_1&u_2&u_3&\omit&v_1&v_2&v_3&\cr
 \noalign{\hreglon}
 e_1&& e_1&0&&u_1&u_2&u_3&&0&0&0&\cr
 e_2&&0&e_2&&0&0&0&&v_1&v_2&v_3&\cr
 &\multispan{12}{\hregletafill}\cr
 u_1&&0&u_1&&0&v_3&-v_2&&-e_1&0&0&\cr
 u_2&&0&u_2&&-v_3&0&v_1&&0&-e_1&0&\cr
 u_3&&0&u_3&&v_2&-v_1&0&&0&0&-e_1&\cr
 &\multispan{12}{\hregletafill}\cr
 v_1&&v_1&0&&-e_2&0&0&&0&u_3&-u_2&\cr
 v_2&&v_2&0&&0&-e_2&0&&-u_3&0&u_1&\cr
 v_3&&v_3&0&&0&0&-e_2&&u_2&-u_1&0&\cr
 \noalign{\hreglon}}}
\]
\caption{{\vrule width 0pt height 15pt}
Multiplication table in a canonical basis of the Zorn matrix algebra.}
\label{ta:good_basis}
\end{table}
\end{example}

\begin{remark}\label{re:canonical_basis}
Given the Zorn matrix algebra and its Peirce decomposition
\begin{equation}\label{eq:Peirce}
\cC=\FF e_1\oplus\FF e_2\oplus\cU\oplus\cV,
\end{equation}
where $\cU=\espan{u_1,u_2,u_3}=e_1\cdot\cC\cdot e_2$, and $\cV=\espan{v_1,v_2,v_3}=e_2\cdot \cC\cdot e_1$, the trilinear map $\cU\times \cU\times \cU\rightarrow \FF$, $(x,y,z)\mapsto n(x,y\cdot z)$, is alternating and nonzero. Then, given any basis $\{\tilde u_1,\tilde u_2,\tilde u_3\}$ of $\cU$ with $n(\tilde u_1,\tilde u_2\cdot \tilde u_3)=1$, its dual basis relative to $n$ in $\cV$ is $\{\tilde v_1=\tilde u_2\cdot\tilde u_3,\tilde v_2=\tilde u_3\cdot\tilde u_1,\tilde v_3=\tilde u_1\cdot\tilde u_2\}$, and $\{e_1,e_2,\tilde u_1,\tilde u_2,\tilde u_3,\tilde v_1,\tilde v_2,\tilde v_3\}$ is another canonical basis, that is, it has the same multiplication table.
\end{remark}

The subalgebra $\FF e_1\oplus\FF e_2$ is isomorphic to the algebra $\FF\times\FF$, the subalgebra $\FF e_1\oplus\FF e_2\oplus\FF u_1\oplus\FF v_1$ is isomorphic to the algebra $M_2(\FF)$ of two by two matrices. The algebras $\FF\times\FF$, $M_2(\FF)$ and Zorn matrix algebra exhaust, up to isomorphism, the Hurwitz algebras with isotropic norm. Together with the ground field, these are called the \emph{split Hurwitz algebras}.

We will make use several times of the following helpful result. Throughout the paper, an \emph{idempotent} of an algebra $(\cA,\ast)$ is a \emph{nonzero} element $e\in\cA$ such that $e\ast e=e$.

\begin{lemma}\label{le:useful}
Let $(\cC,\cdot,n)$ be the Zorn matrix algebra over a field $\FF$. Then,
\begin{itemize}
\item Any idempotent of $(\cC,\cdot,n)$, different from the unity $1$, is conjugate to $e_1$. (That is, there is an automorphism $\varphi\in\Aut(\cC,\cdot,n)$ such that $\varphi(e)=e_1$.)
\item Any nonzero element $x\in\cC$ with $x^{\cdot 2}=0$ is conjugate to $u_1$.
\end{itemize}
\end{lemma}
\begin{proof}
The elements $e$ and $1-e$ are orthogonal idempotents, so there is an isomorphism $\varphi:\FF e_1\oplus\FF e_2\rightarrow \FF 1\oplus\FF e$ such that $\varphi(e_1)=e$. By means of the Cayley-Dickson doubling process (as, for instance, in \cite[Corollary 1.7.3]{SV} or \cite[proof of Corollary 4.7]{EK13}), $\varphi$ can be extended to an automorphism of $(\cC,\cdot,n)$.

Now, if $x^{\cdot 2}=0$, let $y\in\cC$ with $n(x,y)=-1$ and $n(1,y)=0$.  Changing $y$ by $y-\frac{1}{n(y)}x$ we may assume $n(y)=0$ too. Then $\bar x=-x$, $\bar y=-y$ and $1=-n(x,y)1=x\cdot y+y\cdot x$. It follows that $f=x\cdot y$ and $1-f=y\cdot x$ are orthogonal idempotents. The assignment $e_1\mapsto f$, $e_2\mapsto 1-f$, $u_1\mapsto x$, $v_1\mapsto y$, gives an isomorphism $\FF e_1\oplus\FF e_2\oplus\FF u_1\oplus\FF v_1\, \bigl(\cong M_2(\FF)\bigr)$ onto the subalgebra generated by $x$ and $y$ that extends to an automorphism of $(\cC,\cdot,n)$.
\end{proof}

\medskip

\begin{df}\label{df:symmetric}
A composition algebra $(\cS,*,n)$ is said to be \emph{symmetric} if the polar form of the norm is associative:
\begin{equation}\label{eq:n_invariant}
n(x*y,z)=n(x,y*z)
\end{equation}
for any $x,y,z\in\cS$. This is equivalent to the condition
\begin{equation}\label{eq:xyx}
(x*y)*x=n(x)=x*(y*x)
\end{equation}
for any $x,y\in\cC$.
\end{df}

If $e$ is an idempotent of a symmetric composition algebra $(\cS,*,n)$, then the map 
\begin{equation}\label{eq:e_tau}
\tau: x\mapsto e*(e*x)=n(e,x)e-x*e
\end{equation} 
is an automorphism of $(\cS,*,n)$ such that $\tau^3=\id$ (see \cite[Theorem 2.5]{EP96}). Moreover, the new multiplication on $\cS$ given by 
\begin{equation}\label{eq:exye}
x\cdot y=(e*x)*(y*e)
\end{equation}
makes $(\cS,\cdot,n)$ a Hurwitz algebra with unity $e$, and the original multiplication is recovered as
\begin{equation}\label{eq:Petersson}
x*y=\tau(\bar x)\cdot\tau^2(\bar y),
\end{equation}
for any $x,y\in\cS$. Moreover, $\tau$ is also an automorphism of $(\cS,\cdot,n)$.

\begin{proposition}\label{pr:fix_cent}
Let $e$ be an idempotent of a symmetric composition algebra $(\cS,*,n)$, let $\tau$ be the automorphism defined in \eqref{eq:e_tau}, and let $(\cS,\cdot,n)$ be the Hurwitz algebra with multiplication given in \eqref{eq:exye}. Then,
\begin{romanenumerate} 
\item the subalgebra of fixed points by $\tau$:
\[
\Fix(\tau)\bydef \{x\in\cS: \tau(x)=x\}
\]
coincides with the centralizer of $e$:
\[
\Centr_{(\cS,*,n)}(e)\bydef \{x\in\cS: e*x=x*e\}.
\]
\item The centralizer in the group scheme of automorphisms of $(\cS,\cdot,n)$ of $\tau$ coincides with the stabilizer of $e$ in the group scheme of automorphisms of $(\cS,*,n)$:
\[
\CCentr_{\AAut(\cS,\cdot,n)}(\tau)=\SStab_{\AAut(\cS,*,n)}(e).
\]
\end{romanenumerate}
\end{proposition}
\begin{proof}
For $x\in\cS$, $\tau(x)=x$ if and only if $e*(e*x)=x$, if and only if $(e*(e*x))*e=x*e$, if and only if (because of \eqref{eq:xyx}) $e*x=x*e$. This proves the first part.

If $\varphi\in\Aut(\cS,*,n)$ fixes $e$: $\varphi(e)=e$, then clearly $\varphi$ is an automorphism of $\Aut(\cS,*,n)$ that commutes with $\tau$. Conversely, if $\varphi\in\Aut(\cS,\cdot,n)$ and $\varphi\tau=\tau\varphi$, then by \eqref{eq:Petersson}, $\varphi\in\Aut(\cS,*,n)$, and $\varphi(e)=e$ because $e$ is the unity of $(\cS,\cdot,n)$. Besides, all of this is functorial, so it is valid at the level of group schemes.
\end{proof}

\begin{df}\label{df:Petersson}
Given a Hurwitz algebra $(\cC,\cdot,n)$ over a field $\FF$ and an automorphism $\tau\in\Aut(\cC,\cdot,n)$ with $\tau^3=\id$, the new algebra defined on $\cC$ by means of \eqref{eq:Petersson}, and with the same norm, is called a \emph{Petersson algebra}, and denoted by $\cC_\tau$. 
\end{df}
This multiplication appeared for the first time in \cite{Petersson}. Any Petersson algebra is a symmetric composition algebra. If $\tau$ is the identity automorphism: $\tau=\id$, the Petersson algebra $\cC_{\id}$ is called the \emph{para-Hurwitz algebra} associated to the Hurwitz algebra $(\cC,\cdot,n)$. We will talk about para-quadratic algebras in dimension $2$, para-quaternion algebras in dimension $4$, and para-Cayley (or para-octonion) algebras in dimension $8$.

Usually the multiplication in a para-Hurwitz algebra will be denoted by $\bullet$: $x\bullet y=\bar x\cdot\bar y$. The unity $1$ of $\cC$ becomes an idempotent of $\cC_{\id}$ that satisfies $1\bullet x=x\bullet 1=\bar x\,\Bigl(=n(x,1)1-x\Bigr)$. Idempotents with this property are called \emph{para-units}. Any para-unit lies in the commutative center $K(\cC,\bullet)\bydef\{x\in\cC: x\bullet y=y\bullet x\ \forall y\in\cC\}$, which is the whole $\cC$ if $\dim_\FF\cC=1$ or $2$, and equals $\FF 1$ otherwise. Hence there is a unique para-unit if the dimension is $4$ or $8$. In particular, this implies that the group scheme of automorphisms $\AAut(\cC,\cdot,n)$ and $\AAut(\cC,\bullet,n)$ coincide and that any form of a para-Hurwitz algebra (i.e., any algebra that becomes isomorphic to a para-Hurwitz algebra after an extension of scalars)  is itself para-Hurwitz (if the dimension is $4$ or $8$).

\smallskip

There is a natural order $3$ automorphism of the Zorn matrix algebra:
\begin{equation}\label{eq:tau_st}
\tau_{st}\left(\begin{pmatrix} \alpha&(\mu_1,\mu_2,\mu_3)\\ (\nu_1,\nu_2,\nu_3)&\beta\end{pmatrix}\right) = \begin{pmatrix} \alpha&(\mu_3,\mu_1,\mu_2)\\ (\nu_3,\nu_1,\nu_2)&\beta\end{pmatrix}.
\end{equation}

\begin{df}[\cite{EP96}]\label{df:Okubo}
Let $(\cC,\cdot,n)$ be the Zorn matrix algebra, and let $\tau_{st}$ be its order $3$ automorphism in \eqref{eq:tau_st}. The Petersson algebra $\cC_{\tau_{st}}$ is called the \emph{split Okubo algebra}.

The forms of the split Okubo algebra are called \emph{Okubo algebras}.
\end{df}

This is not the original definition of these algebras given by Okubo in \cite{O78}.

\begin{theorem}[\cite{OkuboOsborn1,OkuboOsborn2,EP96}]\label{th:paraH_or_Okubo}
Any symmetric composition algebra is either a form of a para-Hurwitz algebra (hence a para-Hurwitz algebra if the dimension is $\neq 2$), or an Okubo algebra.
\end{theorem}

The proof in \cite{EP96} works as follows. Let $(\cS,*,n)$ be a symmetric composition algebra. Then either it contains an idempotent or it contains an idempotent after a cubic field extension of degree $3$ (\cite[((34.10)]{KMRT}). Assuming the existence of an idempotent, the arguments above show that the symmetric composition algebra is a Petersson algebra. Then in \cite{EP96} it is shown how to find, assuming the field is algebraically closed, an idempotent such that either the automorphism $\tau$ in \eqref{eq:e_tau} is the identity (i.e., $e$ is a para-unit) or the dimension is $8$ and $\tau$ is, up to conjugation, the automorphism $\tau_{st}$ in \eqref{eq:tau_st}.

\bigskip


\section{Zorn matrix algebra and automorphisms}\label{se:Zorn_autos}

Let us consider the Zorn matrix algebra $(\cC,\cdot,n)$, as defined in Example \ref{ex:Zorn}. The stabilizer in $\AAut(\cC,\cdot,n)$ of the orthogonal idempotents $e_1$ and $e_2$:
\begin{equation}\label{eq:Stabilizer_eis}
\{\varphi\in\AAut(\cC,\cdot,n) : \varphi(e_i)=e_1,\ i=1,2\}
\end{equation}
is isomorphic to the special linear group $\SSL(\cU)\simeq\SSL_3$ (see \cite{Jac58} for the rational points, but the arguments are valid in general), where $\cU$ is the Peirce component in \eqref{eq:Peirce}, with the action of $f\in\SSL(\cU)$ given by
\[
f\cdot\begin{pmatrix}\alpha&u\\ v&\beta\end{pmatrix} = 
\begin{pmatrix} \alpha&f(u)\\ (f^t)^{-1}(v)&\beta\end{pmatrix},
\]
where $f^t\in\SSL(\cV)$ is the adjoint of $f$ relative to the scalar product $(.\mid.)$.
The torus $\bfT$ consisting of the diagonal matrices in $\SSL(\cU)$, relative to the basis $\{u_1,u_2,u_3\}$ in Table \ref{ta:good_basis}, is a split maximal torus of $\AAut(\cC,\cdot,n)$.

In the same vein, $\{d\in\Der(\cC,\cdot,n): d(e_i)=0,\ i=1,2\}$ is a Lie subalgebra of $\Der(\cC,\cdot,n)$ isomorphic to $\frsl(\cU)\simeq\frsl_3(\FF)$, and its diagonal subalgebra is a Cartan subalgebra of $\Der(\cC,\cdot,n)$.

The weights of the action of $\bfT$ on $\cC$ are $\{0,\pm\veps_i:1\leq i\leq 3\}$ (additive notation), where for $f\in \bfT$, $f=\diag(\alpha_1,\alpha_2,\alpha_3)$ (with $\alpha_1\alpha_2\alpha_3=1$), $\veps_i(f)=\alpha_i$. Hence $\veps_1+\veps_2+\veps_3=0$.

The root system in $\Der(\cC,\cdot,n)$ is 
\[
\Phi=\{\pm\veps_i,\,\veps_i-\veps_j:1\leq i\neq j\leq 3\}
\] 
(the root system of type $G_2$):
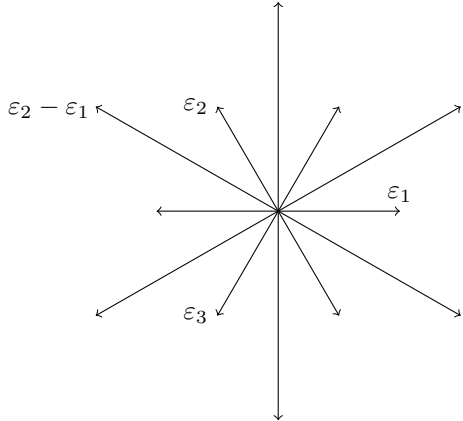
\[
\tik{0.8}{%
    \foreach\ang in {60,120,...,360}{
     \draw[->] (0,0) -- (\ang:2cm);
    }
    \foreach\ang in {30,90,...,330}{
     \draw[->] (0,0) -- (\ang:3.464cm);
    }
    \node[anchor=south] at (2,0) {$\veps_1$}; 
    \node[anchor=east] at (120:2cm) {$\veps_2$};
    \node[anchor=east] at (240:2cm) {$\veps_3$};
    \node[anchor=east] at (150:3.4cm) {$\veps_2-\veps_1$};
}
\]
Then $\Delta=\{\veps_1,\veps_2-\veps_1\}$ is a set of simple roots. The associated set of positive roots is $\Phi^+=\{\veps_1,\veps_2,-\veps_3,\veps_2-\veps_1,\veps_2-\veps_3,\veps_1-\veps_3\}$. 

With $\FF=\CC$ (the field of complex numbers), a concrete Chevalley basis is computed in \cite[4.4]{EK13}:
\[
\begin{split}
x_{\veps_i-\veps_j}&=[L_{u_i},R_{v_j}]=E_{ij}\in\frsl_3(\FF)\subseteq \Der(\cC,\cdot,n),\ 1\leq i\neq j\leq 3,\\
 x_{\veps_i}&=d_{e_1,u_i},\quad x_{-\veps_i}=-d_{e_2,v_i},\\
 h_1&=2E_{11}-E_{22}-E_{33},\quad h_2=E_{22}-E_{11},
\end{split}
\]
where $E_{ij}$ denotes the $3\times 3$ matrix with $1$ in the $(i,j)$-position and $0$'s elsewhere, $L_x$ (respectively $R_x$) denotes the left (resp. right) multiplication by $x$, and $d_{x,y}=[L_x,L_y]+[L_x,R_y]+[R_x,R_y]=\ad_{[x,y]^{\cdot}}+3[L_x,R_y]$ (with $\ad_x(y)=[x,y]^\cdot=x\cdot y-y\cdot x$ for any $x,y\in\cC$).

The $\ZZ$-span of the canonical basis is an admissible lattice, that is, it is invariant under the action of $\frac{x_\alpha^n}{n!}$ for each root $\alpha$. Actually, for each long root $\alpha=\veps_i-\veps_j$, $x_\alpha^2=0$, and for each short root $\alpha=\veps_i$ (resp. $-\veps_i$), $x_{\alpha}^3=0$, and $x_{\alpha}^2$ takes $v_i$ to $2u_i$ (resp. $u_i$ to $2v_i$), and kills all the other elements in the canonical basis. Hence, for an indeterminate $T$, $\exp(Tx_\alpha)$ is an automorphism of $(\cC_{\ZZ[T]},\cdot,n)$, where $\cC_\ZZ$ is the $\ZZ$-span of the canonical basis, and $\cC_{\ZZ[T]}=\cC_\ZZ\otimes_\ZZ\ZZ[T]$. 

As usual, specializing $T$ to elements in an arbitrary  algebra $R\in\AlgF$, it makes sense to consider the induced automorphisms, denoted by $\exp(tx_\alpha)$, $t\in R$, thus obtaining group homomorphisms $\bfG_a\rightarrow \AAut(\cC,\cdot,n)$ for each root $\alpha$, where $\bfG_a$ denotes the additive group ($\bfG(R)$ is the additive group of $R$ for any $R\in\AlgF$). Recall \cite[Lemma 7.3.3]{Springer} that for each connected, reductive, algebraic group $\bfG$ with a maximal split torus $\bfT$, and for each root $\alpha$ relative to this torus, there is an essentially unique homomorphism of algebraic groups $u_\alpha:\bfG_a\rightarrow \bfG$, satisfying some natural restrictions. Here $u_\alpha$ is the map $t\mapsto \exp(tx_\alpha)$ above. If $\bfU_\alpha$ denotes the image of $u_\alpha$, then $\bfG$ is generated by $\bfT$ and the $\bfU_\alpha$'s. (see \cite[8.1.1]{Springer}). Actually, over an algebraically closed field, the group of rational points of the $\bfU_\alpha$'s generate the group of rational points of $\bfG$ as an abstract group.

If the characteristic of $\FF$ is $3$, all these automorphisms $\exp(tx_\alpha)$ (for $t\ne 0$) have order $3$ and will play an important role in Theorem \ref{th:char3}.

Let $\bfB$ be the standard Borel subgroup associated to our set $\Delta$ of simple roots. Then $\bfB$ has dimension $8$ and it is generated by $\bfT$ and the $\bfU_{\alpha}$'s, with $\alpha\in\Phi^+$, and its unipotent radical $\bfU$ is generated by the $\bfU_{\alpha}$'s ($\alpha\in\Phi^+$). Also, for any simple root $\gamma\in\Delta$, we will use the corresponding parabolic subgroup $\bfP_{\{\gamma\}}$. Recall (see \cite[\S 8.4]{Springer}) that $\bfP_{\{\gamma\}}$ is the semidirect product of the Levi subgroup $\bfL_{\{\gamma\}}$, generated by $\bfT$ and $\bfU_{\pm\gamma}$, and its unipotent radical $\textup{R}_u(\bfP_{\{\gamma\}})$, generated by the $\bfU_{\alpha}$'s, with $\alpha\in\Phi^+\setminus\{\gamma\}$. The derived subgroup $[\bfP_{\{\gamma\}},\bfP_{\{\gamma\}}]=[\bfL_{\{\gamma\}},\bfL_{\{\gamma\}}]\ltimes \textup{R}_u(\bfP_{\{\gamma\}})$ is eight-dimensional, with $[\bfL_{\{\gamma\}},\bfL_{\{\gamma\}}]$ being isomorphic to $\SSL_2$ (see Remark \ref{re:Levi}).

\bigskip


\section{Para-Cayley algebras as Petersson algebras}\label{se:paraCayley}

The order $3$ automorphisms of a Cayley algebra whose associated Petersson algebra is para-Cayley are of a very specific nature. These automorphisms were known to Okubo\footnote{Private communication (June 2013)}.

\begin{lemma}\label{le:tau_w}
Let $w$ be an element of a Cayley algebra $(\cC,\cdot,n)$ such that $w^{\cdot 3}=1$. Then the map $x\mapsto w\cdot x\cdot w^{\cdot 2}$ is an automorphism of $(\cC,\cdot,n)$. It is the identity if $w\in\FF 1$, and its order is $3$ otherwise.
\end{lemma}
\begin{proof}
For any $x,y\in \cC$,
\[
\begin{aligned}
\bigl(w\cdot x\cdot w^{\cdot 2}\bigr)&\cdot\bigl(w\cdot y\cdot w^{\cdot 2}\bigr)&&\\
&=
w\cdot\Bigl((x\cdot w^{\cdot 2})\cdot (w\cdot y\cdot w)\Bigr)\cdot w&&\text{(Middle Moufang Identity)}\\
&=w\cdot\Bigl(\bigl(\bigl((x\cdot w^{\cdot 2})\cdot w\bigl)\cdot y\bigl)\cdot w\Bigl)\cdot w&&
\text{(Right Moufang Identity)}\\
&=w\cdot\bigl((x\cdot y)\cdot w\bigr)\cdot w=w\cdot(x\cdot y)\cdot w^{\cdot 2}, &&
\end{aligned}
\]
where the Moufang identities (\cite{Schafer}{p.28}) have been used, together with the fact that any two elements generate an associative subalgebra. This also shows that the order of this automorphism is $1$ if $w\in\FF 1$ and $3$ otherwise.
\end{proof}

\begin{theorem}\label{th:tau_w}
Let $\tau$ be an order $3$ automorphism of a Cayley algebra $(\cC,\cdot,n)$. Then the Petersson algebra $\cC_\tau$ is para-Cayley if and only if there is an element $w\in\cC\setminus \FF 1$ with $w^{\cdot 2}+w+1=0$ such that $\tau(x)=w\cdot x\cdot w^{\cdot 2}$ for any  $x\in\cC$. 

In this case, the element $w$ is the para-unit of $\cC_\tau$.
\end{theorem}
\begin{proof}
Let $\cC_\tau=(\cC,*,n)$, where $x*y=\tau(\bar x)\cdot\tau^2(\bar y)$ for any $x,y\in\cC$. Assume first that $\cC_\tau$ is para-Cayley, and let $e$ be its para-unit. Since $\tau$ is  also an automorphism of $\cC_\tau$, $\tau(e)=e$. Also $n(e)=1$ and $e\not\in\FF 1$ (otherwise $\tau$ would be the identity).

For any $x\in\cC$,
\[
x*e=\begin{cases} n(e,x)e-x,\quad\text{because $e$ is the para-unit,}\\
 \tau(\bar x)\cdot\tau^2(\bar e)=\tau(\bar x)\cdot \bar e.&
 \end{cases}
\]
Therefore 
\[
n(e,x)e^{\cdot 2}-x\cdot e=(x*e)\cdot e=(\tau(\bar x)\cdot \bar e)\cdot e=\tau(\bar x),
\]
and we get, for any $x\in\cC$,
\[
\tau(x)=n(e,\bar x)e^{\cdot 2}-\bar x\cdot e.
\]
In particular $e=n(e,\bar e)e^{\cdot 2}-1$ and, since $e^{\cdot 2}-n(e,1)e+n(e)1=0$ and $1$ and $e$ are linearly independent, we conclude that $n(e,1)=-1=n(e,\bar e)$, so $e^{\cdot 2}+e+1=0$ and
\[
\tau(x)=n(e,\bar x)e^{\cdot 2}-\bar x\cdot e=(e\cdot x+\bar x\cdot \bar e)\cdot e^{\cdot 2} -\bar x\cdot e= e\cdot x\cdot e^{\cdot 2},
\]
for any $x\in\cC$.

Conversely, for $w\in\cC\setminus \FF 1$ with $w^{\cdot 2}+w+1=0$ such that $\tau(x)=w\cdot x\cdot w^{\cdot 2}$, and  with $x*y=\tau(\bar x)\cdot \tau^2(\bar y)$, one gets
\[
\begin{split}
w*x&=\tau(\bar w)\cdot\tau^2(\bar x)=\bar w\cdot w^{\cdot 2}\cdot\bar x\cdot w\\
 &=w\cdot\bar x\cdot w=\bigl(n(w,x)1-x\cdot\bar w\bigr)\cdot w=n(w,x)w-x,
 \end{split}
\]
and, in the same vein, $x*w=n(x,w)w-x=w*x$. Therefore, $w$ is a para-unit of $\cC_\tau$, so that $\cC_\tau$ is para-Cayley.
\end{proof}

\begin{remark}\label{re:tau_w}
Let $(\cC,\cdot,n)$ be a Cayley algebra and let $w\in\cC\setminus\FF 1$ be an element such that $w^{\cdot 2}+w+1=0$. Denote by $\tau_w$ the order $3$ automorphism $x\mapsto w\cdot x\cdot w^{\cdot 2}$.
\begin{itemize}
\item If $\charac\FF\ne 3$, then the subalgebra generated by $w$ is $\alg\langle w\rangle=\FF 1+\FF w$, which is isomorphic to the quadratic \'etale algebra $\FF[X]/(X^2+X+1)$. For any element $x\in\cC$ orthogonal to $\alg\langle w\rangle$,
\[
\tau_w(x)=w\cdot x\cdot w^{\cdot 2}=w\cdot x\cdot \bar w=x\cdot \bar w^{\cdot 2}=x\cdot w.
\]
But $n(1-w)=3\ne 0$, so the right multiplication by $1-w$ is a bijection, and hence $\tau_w(x)\ne x$. Therefore, the subalgebra $\Fix(\tau_w)$ of the elements fixed by $\tau_w$ coincides with $\alg\langle w\rangle=\FF 1+\FF w$.

If $\tilde w$ is another element in $\cC\setminus \FF 1$ with $\tilde w^{\cdot 2}+\tilde w+1=0$, there there is an isomorphism $\varphi:\alg\langle w\rangle\rightarrow \alg\langle \tilde w\rangle$ with $\varphi(w)=\tilde w$. By means of the Cayley-Dickson doubling process, this isomorphism $\varphi$ can be extended to an automorphism of $(\cC,\cdot,n)$, also denoted by $\varphi$. Then $\tau_{\tilde w}=\varphi\circ\tau_w\circ\varphi^{-1}$, so that all these order $3$ automorphisms are conjugate.

\item If $\charac\FF=3$, then $0=w^{\cdot 2}+w+1=(w-1)^{\cdot 2}$, so the nonzero element $u=w-1$ satisfies $u^{\cdot 2}=0$, and hence $(\cC,\cdot,n)$ is the split Cayley algebra and, by Lemma \ref{le:useful}, there is a canonical basis with $u=u_1$, so that $w=1+u_1$, $w^{\cdot 2}=1-u_1$. In particular, again all these order $3$ automorphisms are conjugate. Computing in this canonical basis one gets:
\[
\begin{split}
&1\xrightarrow{\tau_w-\id} 0,\\
&v_1\xrightarrow{\tau_w-\id}(1+u_1)\cdot v_1\cdot(1-u_1)-v_1=-e_1+e_2+u_1\\
&\hspace*{1in}
\xrightarrow{\tau_w-\id}(1+u_1)\cdot e_1\cdot(1-u_1)-e_1=-u_1\xrightarrow{\tau_w-\id} 0,\\
&u_2\xrightarrow{\tau_w-\id}(1+u_1)\cdot u_2\cdot(1-u_1)-u_2=-v_3\xrightarrow{\tau_w-\id} 0,\\
&u_3\xrightarrow{\tau_w-\id}(1+u_1)\cdot u_3\cdot(1-u_1)-u_3=v_2\xrightarrow{\tau_w-\id} 0,
\end{split}
\]
Thus, the Segre symbol of the nilpotent linear map $\tau_w-\id$ is $(3,2^2,1)$.
\end{itemize}
\end{remark}

\begin{corollary}\label{co:idempotents_paraCayley}
Let $(\cC,\cdot,n)$ be a Cayley algebra and $(\cC,\bullet,n)$ the associated para-Cayley algebra ($x\bullet y=\bar x\cdot \bar y$ for any $x,y\in\cC$). Then
\[
\{\text{idempotents of $(\cC,\bullet,n)$}\}=\{1\}\cup\{w\in\cC\setminus\FF 1: w^{\cdot 2}+w+1=0\},
\]
and all the idempotents, with the exception of the para-unit $1$, are conjugate under $\Aut(\cC,\cdot,n)=\Aut(\cC,\bullet,n)$.
In particular,
\begin{itemize}
\item
If either $\charac\FF\ne 3$ and $(\cC,\bullet,n)$ contains a subalgebra isomorphic to the quadratic \'etale algebra $\cK=\FF[X]/(X^2+X+1)$, or $\charac\FF=3$ and $(\cC,\cdot,n)$ is split, then $(\cC,\bullet,n)$ contains the para-unit and a unique conjugacy class of other idempotents.
\item Otherwise, the only idempotent of $(\cC,\bullet,n)$ is its para-unit.
\end{itemize}
\end{corollary}
\begin{proof}
If $w\in\cC\setminus\FF 1$ satisfies $w\bullet w=w$, then $n(w)=1$ and $w^{\cdot 2}=\bar w=n(1,w)1-w$. As $w^{\cdot 2}-n(1,w)w+n(w)1=0$, it follows that $n(1,w)=-1$ and $w^{\cdot 2}+w+1=0$. Conversely, if $w^{\cdot 2}+w+1=0$ for $w\in\cC\setminus\FF 1$, then $n(1,w)=-1$ so $w^{\cdot 2}=\bar w$, and $w\bullet w=w$. Now the arguments in Remark \ref{re:tau_w} apply.
\end{proof}

\smallskip

We finish this section by looking at the (easier) situation in dimension $2$ and $4$. (See also \cite[Theorem 3.2]{EP96}.)

\begin{proposition}\label{pr:order_3_Q}
Let $\tau$ be an order $3$ automorphism of a quaternion algebra $(\cQ,\cdot,n)$. Then there exists an element $w\in\cQ\setminus\FF 1$ with $w^{\cdot 2}+w+1=0$ such that $\tau(x)=w\cdot x\cdot w^{\cdot 2}$ for any $x\in\cQ$. All these automorphisms are conjugate.
\end{proposition}
\begin{proof}
By the Noether-Skolem Theorem, there is an invertible element $a\in\cQ\setminus\FF 1$ such that $\tau(x)=a\cdot x\cdot a^{-1}$ and $a^{\cdot 3}=\lambda 1$, $\lambda\in\FF^\times$. Then $a^{\cdot 4}\in\FF^\times a$, so with $w=\frac{1}{n(a^{\cdot 2})}a^{\cdot 4}$, $\tau(x)=w\cdot x\cdot w^{-1}$ and $n(w)=1$. Besides, $w^{\cdot 3}=\frac{1}{n(a)^6}a^{\cdot 12}=\frac{1}{\lambda^4}(\lambda 1)^{\cdot 4}=1$. Then $w^{\cdot 2}=\bar w\cdot w^{\cdot 3}=\bar w$ and, as in the proof of Corollary \ref{co:idempotents_paraCayley}, we get $w^{\cdot 2}+w+1=0$. The conjugation of these automorphisms follows as in Remark \ref{re:tau_w}.
\end{proof}

\begin{remark}\label{re:Q}
Under the conditions of Proposition \ref{pr:order_3_Q}, and as in Theorem \ref{th:tau_w}, the Petersson algebra $\cQ_\tau$ is a para-quaternion algebra with para-unit $w$. Since the norm is the same $n$, $\cQ_\tau$ is isomorphic to $(\cQ,\bullet,n)$.
\end{remark}

\begin{corollary}\label{co:idempotentes_paraQ}
Let $(\cQ,\cdot,n)$ be a quaternion algebra and $(\cQ,\bullet,n)$ the associated para-quaternion algebra. Then
\[
\{\text{idempotents of $(\cQ,\bullet,n)$}\}=\{1\}\cup\{w\in\cQ\setminus\FF 1: w^{\cdot 2}+w+1=0\},
\]
and all the idempotents, with the exception of the para-unit $1$, are conjugate under $\Aut(\cQ,\cdot,n)=\Aut(\cQ,\bullet,n)$.
In particular,
\begin{itemize}
\item
If either $\charac\FF\ne 3$ and $(\cQ,\bullet,n)$ contains a subalgebra isomorphic to the quadratic \'etale algebra $\cK=\FF[X]/(X^2+X+1)$, or $\charac\FF=3$ and $(\cQ,\cdot,n)$ is split, then $(\cQ,\bullet,n)$ contains the para-unit and a unique conjugacy class of other idempotents.
\item Otherwise, the only idempotent of $(\cQ,\bullet,n)$ is its para-unit.
\end{itemize}
\end{corollary}

Finally, the next remark settles the situation for two-dimensional symmetric composition algebras. 

\begin{remark}\label{re:dim_2}
The automorphism group $\Aut(\cK,\cdot,n)$ for a two-dimensional Hurwitz algebra (i.e., an \'etale quadratic algebra) is the cyclic group of order $2$. In particular, it does not contain order $3$ automorphisms. Hence, if a two-dimensional symmetric composition algebra $(\cS,*,n)$ contains an idempotent, this idempotent is a para-unit, so that $(\cS,*,n)$ is a para-quadratic algebra. Moreover,
\begin{itemize}
\item 
If $\charac\FF =3$, the para-unit is the only idempotent of a para-quadratic algebra.
\item
If $\charac\FF\neq 3$, the only para-quadratic algebra containing other idempotents is, up to isomorphism, the para-quadratic algebra $(\cK,\bullet,n)$ associated to the quadratic \'etale algebra $(\cK,\cdot,n)$, with $\cK=\FF 1\oplus\FF w$ and $w^{\cdot 2}+w+1=0$. Then
\[
\{\text{idempotents of $(\cQ,\bullet,n)$}\}=\{1,w,w^{\cdot 2}\},
\]
and $\Aut(\cK,\bullet,n)$ is the symmetric group $\sfS_3$ of degree $3$ (see \cite[(34.5)]{KMRT}), which acts permuting the three idempotents. Actually, the affine group scheme of automorphisms $\AAut(\cK,\bullet,n)$ is the corresponding constant group scheme, also denoted by $\sfS_3$.
\end{itemize}
\end{remark}

\bigskip

\section{Order $3$ elements in $G_2$ and idempotents, $\charac\FF\ne 3$}\label{se:char_not_3}

Over an algebraic closed field of characteristic $\ne 3$, any order $3$ element of the algebraic group of type $G_2$ is semisimple and contained in a torus. By conjugacy of the maximal tori, any such element $\tau$ is conjugate to an element in the torus consisting of the diagonal matrices in the subgroup $\SSL_3(\FF)$ in \eqref{eq:Stabilizer_eis}. Therefore, there is a primitive cubic root $\omega$ of $1$ such that $\tau$ is conjugate to $\diag(\omega,\omega,\omega)$ or $\diag(1,\omega,\omega^2)$ (changing $\omega$ by $\omega^2$ gives a conjugate element). This shows that there are exactly two conjugacy classes of such elements. The same can be deduced from the description of finite order automorphisms in \cite[Ch. VIII]{Kac} in characteristic $0$ and its extension to the modular case by Serre \cite{Serre} (see also \cite[Theorem 6.3]{PL}).

Over arbitrary fields of characteristic $\neq 3$, the situation is not much worse.

\begin{theorem}\label{th:char_not3}
Let $(\cC,\cdot,n)$ be a Cayley algebra over a field $\FF$ of characteristic not $3$, and let $\tau$ be an order $3$ automorphism of $(\cC,\cdot,n)$. Then one of the following conditions holds:
\begin{enumerate}
\item There is an element $w\in\cC\setminus \FF 1$ with $w^{\cdot 2}+w+1=0$ such that
\[
\tau(x)=w\cdot x\cdot w^2
\]
for any $x\in\cC$. In this case, the Petersson algebra $\cC_\tau$ is para-Cayley with para unit $w$, and the subalgebra $\Fix(\tau)$ of the elements fixed by $\tau$ is $\FF 1+\FF w$ (a quadratic \'etale subalgebra). Any two such automorphisms are conjugate in $\Aut(\cC,\cdot,n)$.

\item The subalgebra of fixed elements by $\tau$ is a quaternion subalgebra of $\cC$ containing an element $w\in\cC\setminus\FF 1$ such that $w^{\cdot 2}+w+1=0$. In this case the Petersson algebra $\cC_\tau$ is an Okubo algebra. Any two such automorphisms are conjugate in $\Aut(\cC,\cdot,n)$ if and only if the corresponding quaternion subalgebras are isomorphic. \\
In particular, if $\FF$ contains the cubic roots of $1$, then $\cC$ is the split Cayley algebra, $\cC_\tau$ is the split Okubo algebra, and any two such automorphisms are conjugate.
\end{enumerate}
\end{theorem}
\begin{proof} By \cite[Proposition 3.4 and Theorem 3.5]{EP96}, either $\Fix(\tau)$ is a quadratic \'etale subalgebra and $\cC_\tau$ is para-Cayley, or $\Fix(\tau)$ is a quaternion subalgebra.

In the first case, Theorem \ref{th:tau_w} and Remark \ref{re:tau_w} give us the desired result. In the second case, $\cQ=\Fix(\tau)$ is a quaternion subalgebra and $\cC=\cQ\oplus\cQ\cdot u$ for a nonisotropic element $u$ orthogonal to $\cQ$. Then $\tau(u)=w\cdot u$ for a norm $1$ element $w\in\cQ$ such that $w^{\cdot 3}=1\neq w$. Then $w^{\cdot 2}=\bar w\cdot w^{\cdot 3}=\bar w=n(1,w)1-w$ and, as in the proof of Proposition \ref{pr:order_3_Q}, we have $w^{\cdot 2}+w+1=0$.

If $\tau'$ is another such automorphism and the quaternion subalgebra $\cQ'=\Fix(\tau')$ is isomorphic to $\cQ$, then the orthogonal complements $\cQ^\perp$ and $(\cQ')^\perp$ are isometric, so there is an element $u'\in(\cQ')^\perp$ with $n(u')=n(u)$. As before, $\tau'(u')=w'\cdot u'$ with $w'\in\cQ'$ such that $(w')^{\cdot 2}+w'+1=0$. The isomorphism $\varphi:\FF 1\oplus\FF w\rightarrow \FF 1\oplus\FF w'$ that takes $w$ to $w'$ can be extended to an isomorphism $\varphi:\cQ\rightarrow \cQ'$, and then to an automorphism $\varphi$ of $(\cC,\cdot,n)$, also denoted by $\varphi$, with $\varphi(u)=u'$. It follows that $\tau'=\varphi\circ\tau\circ\varphi^{-1}$, as required.

If $\FF$ contains the cubic roots of $1$, then the subalgebra $\FF 1\oplus\FF w$ is isomorphic to the split quadratic algebra $\FF\times \FF$, so $(\cC,\cdot,n)$ is the split Cayley algebra, and the uniqueness up to conjugacy of $\tau$ shows that there is a canonical basis such that $\tau(u_i)=u_{i+1}$ (indices modulo $3$), so that $\cC_\tau$ is the split Okubo algebra in this case. If $\FF$ does not contain the cubic roots of $1$, this shows that $\cC_\tau$ is an Okubo algebra, as it is so over a field extension.
\end{proof}

\begin{example}\label{ex:real}
Let $(\cC,\cdot,n)$ be the split Cayley algebra over the field of real numbers $\RR$. Then $\cC$ contains subalgebras isomorphic to both the real quaternion division algebra $\HH$ and the split real quaternion algebra $M_2(\RR)$. Both of them contain copies of $\CC$ and hence contain an element $w\not\in\RR 1$ with $w^{\cdot 2}+w+1=0$. Therefore, there are two conjugacy classes of order $3$ automorphisms of $(\cC,\cdot,n)$ such that $\cC_\tau$ is an Okubo algebra.
\end{example}

The centralizers in the affine group scheme of automorphisms $\AAut(\cC,\cdot,n)$ of the automorphisms $\tau$ in Theorem \ref{th:char_not3} will be computed now. Because of Proposition \ref{pr:fix_cent}, these centralizers coincide with the stabilizers of the idempotent $1$ of the symmetric composition algebra (Petersson algebra) $\cC_\tau$.

Some preliminaries are needed first.

Let $\cK$ be a quadratic \'etale subalgebra of a Cayley algebra $(\cC,\cdot,n)$ over a field $\FF$. Then (see \cite{Jac58,Colour-Cayley}), the orthogonal subspace $\cW=\cK^\perp$ is a free left $\cK$-module of rank $3$ endowed with
\begin{itemize}
\item a hermitian nondegenerate form $\sigma:\cW\times\cW\rightarrow\cK$ (i.e., $\sigma$ is $\FF$-bilinear, $\sigma(a\cdot x,y)=a\cdot\sigma(x,y)$ and $\sigma(y,x)=\overline{\sigma(x,y)}$, for any $a\in\cK$ and $x,y\in\cW$, and $\sigma$ induces a $\cK$-linear isomorphism $\cW\rightarrow \Hom_\cK(\cW,\cK)$, $x\mapsto \sigma(.,x)$), and 
\item an anticommutative product $\cW\times\cW\rightarrow \cW$, $(x,y)\mapsto x\times y$, with $(a\cdot x)\times y=\bar a\cdot(x\times y)=x\times(a\cdot y)$ for any $a\in\cK$ and $x,y\in\cW$, 
\end{itemize}
such that, for any $x,y\in\cW$,
\begin{equation}\label{eq:sigma_x}
x\cdot y=-\sigma(x,y)+x\times y.
\end{equation}
The $\cK$-trilinear form $\Phi:\cW\times\cW\times\cW\rightarrow \cK$, given by
\begin{equation}\label{eq:Phi_sigma}
\Phi(x,y,z)=\sigma(x,y\times z),
\end{equation}
is alternating, and satisfies
\begin{equation}\label{eq:Phi_det}
n\bigl(\Phi(x_1,x_2,x_3)\bigr)=\det\Bigl(\sigma(x_i,x_j)\Bigr)
\end{equation}
for any $x_1,x_2,x_3\in\cW$.

Conversely, if $(\cK,\cdot)$ is a quadratic \'etale algebra with norm $n$ and standard conjugation $a\mapsto \bar a$, and $\cW$ is a free left $\cK$-module of rank $3$ endowed with:
\begin{itemize}
\item a nondegenerate hermitian form $\sigma:\cW\times\cW\rightarrow\cK$,
\item a $\cK$-trilinear alternating form $\Phi:\cW\times\cW\times\cW\rightarrow \cK$ satisfying \eqref{eq:Phi_det},
\end{itemize}
then the vector space direct sum $\cC=\cK\oplus\cW$ is a Cayley algebra, where the multiplication $\cdot$ and the norm $n$ are defined as follows:
\begin{itemize}
\item $\cK$ is a subalgebra of $\cC$ and for any $a\in\cK$ and $w\in\cW$, $a\cdot w$ is given by the product of the scalar $a$ on the element $w$ in the $\cK$-module $\cW$, while $w\cdot a\bydef \bar a\cdot w$.
\item $x\cdot y=-\sigma(x,y)+x\times y$ for any $x,y\in\cW$, where $x\times y$ is defined by the equation $\sigma(z,x\times y)=\Phi(z,x,y)$ for any $z\in\cW$.
\item $\cK$ and $\cW$ are orthogonal for $n$, the restriction of $n$ to $\cK$ is the generic norm of $\cK$, and $n(x)=\sigma(x,x)$ for any $x\in\cW$.
\end{itemize}

Assume now that we are in the situation of Theorem \ref{th:char_not3}.(1), so $\charac\FF\neq 3$, and $\tau\in\Aut(\cC,\cdot,n)$ is given by $\tau(x)=w\cdot x\cdot w^{\cdot 2}$, with $w\in\cC\setminus\FF 1$, $w^{\cdot 2}+w+1=0$. Then $\cK=\FF 1+\FF w$ is a quadratic \'etale subalgebra of $(\cC,\cdot,n)$. Hence $\cW=\cK^\perp$ is endowed with the hermitian form $\sigma$ and product $\times$ in \eqref{eq:sigma_x}.

\begin{lemma}\label{le:Stab_w}
Under the conditions above, $\SStab_{\AAut(\cC,\cdot,n)}(w)$ is naturally isomorphic to the special unitary group $\SSU(\cW,\sigma)$.
\end{lemma}
\begin{proof}
Any $\cK$-linear map of $\cW$ of determinant $1$ preserves the alternating $\cK$-trilinear form $\Phi$ in \eqref{eq:Phi_sigma} and hence, if it preserves $\sigma$, it preserves also $\times$ (see \cite{Jac58} if $\charac\FF\neq 2$, but the arguments are valid too if $\charac\FF=2$), so it extends to an automorphism of $(\cC,\cdot,n)$ which restricts to the identity map on $\cK$. Conversely, any automorphism that fixes $w$ restricts to a $\cK$-linear map that preserves $\sigma$ and $\times$, and hence $\Phi$, so it lies in the special unitary group. Besides, all this is functorial, so the result is valid at the level of schemes.
\end{proof}

\begin{theorem}\label{th:tau_Q_u_w}
Let $(\cC,\cdot,n)$ be a Cayley algebra over a field $\FF$ of characteristic $\neq 3$, and let $w\in\cC\setminus \FF 1$ with $w^{\cdot 2}+w+1=0$. Let $\tau$ be the order $3$ automorphism given by $\tau(x)=w\cdot x\cdot w^{\cdot 2}$ for any $x\in\cC$. Then,
\[
\CCentr_{\AAut(\cC,\cdot,n)}(\tau)=\SStab_{\AAut(\cC,\cdot,n)}(w)\,\bigl(\simeq \SSU(\cW,\sigma)\bigr).
\]
\end{theorem}
\begin{proof}
Any $\varphi$ in $\CCentr_{\AAut(\cC,\cdot,n)}(\tau)$ stabilizes $\cK=\Fix(\tau)$ so (locally) we have either $\varphi(w)=w$ or $\varphi(w)=w^{\cdot 2}$, because the group scheme $\AAut(\cK,\cdot)$ is the constant group scheme $\mathsf{C}_2$ (cyclic group of order $2$). But $\varphi(w)\cdot\varphi(x)\cdot\varphi(w)^{\cdot 2}=w\cdot\varphi(x)\cdot w^{\cdot 2}$ for any $x\in\cC\otimes_\FF R$ (if $\varphi$ is a point over $R$), so $w^{-1}\cdot \varphi(w)$ lies in the commutative center of $(\cC,\cdot)$, which is $\FF 1$. Hence $\varphi(w)=w$.
\end{proof}

\smallskip

For automorphisms in Theorem \ref{th:char_not3}.(2), consider first a Cayley algebra $(\cC,\cdot,n)$ over an arbitrary field, and let $\cQ$ be a quaternion subalgebra. Then $\cC=\cQ\oplus\cQ\cdot u$ for some nonisotropic element $u$ orthogonal to $\cQ$. Consider the group scheme $\SSL_1(\cQ)$, whose set of points over an algebra $R\in\AlgF$ is $\SSL_1(\cQ)(R)=\{q\in \cQ\otimes_\FF R: n(q)=1\}$. The scheme $\boldmu_2$ of square roots of unity embeds naturally in $\SSL_1(\cQ)$ as the scalar elements.

\begin{lemma}\label{le:Stab_Q}
Under the conditions above, the stabilizer $\SStab_{\AAut(\cC,\cdot,n)}(\cQ)$ is isomorphic to $\SSL_1(\cQ)\times\SSL_1(\cQ)/\boldmu_2$ (where $\boldmu_2$ embeds diagonally on the product of the two copies of $\SSL_1(\cQ)$).
\end{lemma}
\begin{proof}
For any $q\in\cQ$ of norm $1$, the maps $\psi_q,\varphi_q:\cC\rightarrow \cC$ given by
\[
\begin{aligned}
\psi_q(x)&=q\cdot x\cdot q^{-1}=q\cdot x\cdot \bar q,&\quad \psi_q(x\cdot u)&=(x\cdot q^{-1})\cdot u,\\
\varphi_q(x)&=x,&\varphi_q(x\cdot u)&=(q\cdot x)\cdot u,
\end{aligned}
\]
for $x\in\cQ$, are automorphisms of $(\cC,\cdot,n)$. The same happens for $\cC\otimes_\FF R$ and $q\in\cQ\otimes_\FF R$ of norm $1$, for any algebra $R\in\AlgF$.

Moreover, $\psi_q$ and $\varphi_p$ commute for any $q,p\in\SSL_1(\cQ)$. Therefore, we have the natural homomorphism
\begin{equation}\label{eq:Phi_ssl1Q}
\begin{split}
\Phi:\SSL_1(\cQ)\times\SSL_1(\cQ)&\longrightarrow \SStab_{\AAut(\cC,\cdot,n)}(\cQ)\\
(q,p)\quad&\mapsto\quad \psi_q\circ\varphi_p.
\end{split}
\end{equation}
Its kernel consists of the pairs $(q,p)$ such that $\psi_q\circ\psi_p=\id$, and this forces $\psi_q\vert_\cQ=\id$, so $q\in\boldmu_2$. Then $u=\psi_q\circ\varphi_p(u)=(p\cdot q^{-1})\cdot u$, so $p=q$. Thus the kernel is $\boldmu_2$ embedded diagonally. This shows, in particular, that the dimension of $\SStab_{\AAut(\cC,\cdot,n)}(\cQ)$ is at least $6$.

Finally, $\Phi$ is a quotient map as $\Phi_{\bar\FF}$ is surjective (this follows easily using, for instance, \cite[\S2.1]{SV}) and $\SStab_{\AAut(\cC,\cdot,n)}(\cQ)$ is smooth, because its Lie algebra $\frs=\{\delta\in\Der(\cC,\cdot,n): \delta(\cQ)\subseteq \cQ\}$ has dimension $6$. To check this, consider the homomorphism
\[
\begin{split}
\frs&\longrightarrow \Der(\cQ,\cdot)\simeq \ad(\cQ)\simeq \cQ/\FF 1\\
\delta &\mapsto\quad \delta\vert_{\cQ},
\end{split}
\]
whose kernel consists of those derivations $\delta\in\Der(\cC,\cdot,n)$ such that $\delta(\cQ)=0$. But any such derivation is determined by $\delta(u)$, which belongs to $\{x\in\cQ\cdot u : n(x,u)=0\}$, whose dimension is $3$.
\end{proof}

\begin{remark}\label{re:Levi}
If $\{e_1,e_2,u_1,u_2,u_3,v_1,v_2,v_3\}$ is a canonical basis of the split Cayley algebra $(\cC,\cdot,n)$ over an arbitrary field $\FF$, let $\cQ$ be the quaternion subalgebra spanned by $e_1$, $e_2$, $u_1$, and $v_1$. The group $\SSL_2\simeq\SSL_1(\cQ)$ embeds in $\AAut(\cC,\cdot,n)$ both by means of $q\mapsto \psi_q$ and by $q\mapsto \varphi_q$. With the notation in Section \ref{se:Zorn_autos}, the image is the derived subgroup of the Levi subgroup $\bfL_{\{\veps_1\}}$ in the first case, and of $\bfL_{\{\veps_2-\veps_3\}}$ in the second case.
\end{remark}

\begin{theorem}\label{th:Stab_tau_wxw2}
Let $(\cC,\cdot,n)$ be a Cayley algebra over a field $\FF$ of characteristic $\neq 3$, and let $\tau$ be an order $3$ automorphism of $(\cC,\cdot,n)$ that fixes elementwise a quaternion subalgebra $\cQ$. Let $u\in\cQ^\perp$ be a nonisotropic element and let $w\in\cQ\setminus\FF 1$ with $w^{2}+w+1=0$ such that $\tau(u)=w\cdot u$. Denote by $\cK$ the quadratic \'etale subalgebra $\cK=\FF 1+\FF w$. Then the centralizer $\CCentr_{\AAut(\cC,\cdot,n)}(\tau)$ is isomorphic to $\SSL_1(\cQ)\times\SSL_1(\cK)/\boldmu_2$.
\end{theorem}
\begin{proof}
As $\Fix(\tau)=\cQ$, it follows that the centralizer $\CCentr_{\AAut(\cC,\cdot,n)}(\tau)$ is contained in $\SStab_{\AAut(\cC,\cdot,n)}(\cQ)$. The homomorphism $\Phi$ in \eqref{eq:Phi_ssl1Q} restricts to a homomorphism $\SSL_1(\cQ)\times\SSL_1(\cK)\rightarrow \CCentr_{\AAut(\cC,\cdot,n)}(\tau)$. (Observe  that $\psi_q\circ\tau=\tau\circ\psi_q$ for any $q\in\SSL_1(\cQ)$, but $\varphi_p\circ\tau=\tau\circ\varphi_p$ if and only if $p\cdot w=w\cdot p$, if and only if $p\in\cK$ (and the same over any $R$). Now the same arguments as in the proof of Lemma \ref{le:Stab_Q} apply.
\end{proof}

\smallskip

Due to the relationship between idempotents in symmetric composition algebras and automorphisms of Cayley algebras of order $1$ or $3$, we immediately get the next consequence of Theorem \ref{th:char_not3}. Here $\cK$ will denote the quadratic \'etale algebra $\FF[X]/(X^2+X+1)$, and $\overline{\cK}$ the corresponding para-quadratic algebra.

\begin{corollary}\label{co:idempotents_char_not3}
An  Okubo algebra $(\cO,*,n)$ over a field $\FF$ of characteristic $\neq 3$ contains an idempotent if and only if it contains a subalgebra isomorphic to $\overline{\cK}$. In this case, the centralizer of any idempotent $e$ is a para-quaternion subalgebra containing a subalgebra isomorphic to $\wb{\cK}$, and two idempotents are conjugate if and only if the corresponding quaternion algebras are isomorphic.

In particular, if $\FF$ contains the cubic roots of $1$, then $(\cO,*,n)$ contains a unique conjugacy class of idempotents.
\end{corollary}
\begin{proof}
Let $e$ be an idempotent of the Okubo algebra $(\cO,*,n)$, and let $\tau_e\in\Aut(\cO,*,n)$ be the order $3$ automorphism given by $\tau_e(x)=e*(e*x)$ for any $x\in\cO$ as in \eqref{eq:e_tau}. Consider the Cayley algebra $(\cO,\cdot,n)$, with unity $e$, where $x\cdot y=(e*x)*(y*e)$, as in \eqref{eq:exye}. Then $\tau_e$ is an automorphism of $(\cO,\cdot,n)$ too. By Theorem \ref{th:char_not3}, $\Fix(\tau_e)=\Centr_{(\cO,*,n)}(e)$ is a quaternion subalgebra of $(\cO,\cdot,n)$ (and a para-quaternion subalgebra of $(\cO,*,n)$) containing a subalgebra isomorphic to $\cK$, so that $(\cO,*,n)$ contains a subalgebra isomorphic to $\overline{\cK}$. Moreover, if $f$ is another idempotent, $\tau_f$ the corresponding order $3$ automorphism, and $(\cO,\diamond,n)$ the associated Cayley algebra ($x\diamond y=(f*x)*(y*f)$), then both Cayley algebras $(\cO,\cdot,n)$ and $(\cO,\diamond,n)$ are isomorphic, as their norms are isometric. If the para-quaternion algebras $\Centr_{(\cO,*,n)}(e)$ and $\Centr_{(\cO,*,n)}(f)$ are isomorphic, then by Theorem \ref{th:char_not3} there is an isomorphism $\varphi:(\cO,\cdot,n)\rightarrow (\cO,\diamond,n)$ such that $\varphi\circ\tau_e=\tau_f\circ \varphi$. Note that $\varphi(e)=f$ because $e$ (respectively $f$) is the unity of $(\cO,\cdot,n)$ (resp., of $(\cO,\diamond,n)$). Then, for any $x,y\in\cO$,
\[
\varphi(x*y)=\varphi\bigl(\tau_e(\bar x)\cdot\tau_e^2(\bar y)\bigr)
 =\tau_f(\overline{\varphi(x)})\diamond\tau_f^2(\overline{\varphi(y)})=\varphi(x)*\varphi(y), 
\]
and $\varphi\in \Aut(\cO,*,n)$ too, with $\varphi(e)=f$. 
\end{proof}

Up to isomorphism there is a unique Cayley algebra with isotropic norm: the Zorn matrix algebra (or split Cayley algebra). However, the split Okubo algebra (Definition \ref{df:Okubo}) is not characterized by its norm being isotropic. If the characteristic of $\FF$ is not $3$ and $\FF$ contains the cubic roots of $1$, or if the characteristic of $\FF$ is $3$, then the norm of any Okubo algebra is isotropic (see \cite[Proposition 7.3]{EM93} and \cite{Eld97}).

\begin{theorem}\label{th:O_split_e}
Let $(\cO,*,n)$ be an Okubo algebra over a field $\FF$ of characteristic $\neq 3$. Then $(\cO,*,n)$ is split (i.e., isomorphic to the split Okubo algebra in Definition \ref{df:Okubo}) if and only if $n$ is isotropic and $(\cO,*,n)$ contains an idempotent.
\end{theorem}
\begin{proof}
The split Okubo algebra is the Petersson algebra $\cC_{\tau_{st}}=(\cC,*,n)$, where $(\cC,\cdot,n)$ is Zorn matrix algebra, $\tau_{st}$ is defined in \eqref{eq:tau_st}, and 
\[
x*y=\tau_{st}(\bar x)\cdot\tau_{st}^2(\bar y)
\]
for any $x,y\in\cC$. The norm $n$ is isotropic and the unity $1$ of $(\cC,\cdot,n)$ is an idempotent of $\cC_{\tau_{st}}$.

Conversely, let $(\cO,*,n)$ be an Okubo algebra with isotropic norm $n$ and let $0\neq e=e*e$ be an idempotent. Let $\tau_e\in\Aut(\cO,*,n)$ as in \eqref{eq:e_tau}, and consider the Cayley algebra $(\cO,\cdot,n)$ with unity $e$, where $x\cdot y=(e*x)*(y*e)$ as in \eqref{eq:exye}. Then $x*y=\tau_e(\bar x)\cdot\tau_e^2(\bar y)$, where $\bar x=n(e,x)e-x$ is the `conjugation' relative to $e$. By Theorem \ref{th:char_not3}, $\cQ=\Fix(\tau_e)$ is a quaternion subalgebra of $(\cO,\cdot,n)$ and $\cO=\cQ\oplus\cQ\cdot u$ for a nonisotropic element $u\in\cO$ orthogonal to $\cQ$. Besides, $\tau_e(u)=w\cdot u$ for a norm $1$ element $w\in\cQ$ such that $w^{\cdot 3}=1\neq w$. If the restriction of $n$ to $\cQ$ is isotropic, then $(\cQ,\cdot,n)$ is isomorphic to $M_2(\FF)$, so the restriction $n\vert_\cQ$ is universal. Hence, by multiplying by a suitable $q\in\cQ$, we may asume $n(u)=1$ (this will change the element $w$). Then the multiplication in $(\cO,*,n)$ is completely determined (see \cite[Example 9.4]{CEKT}) so, up to isomorphism, $(\cO,*,n)$ is the split Okubo algebra.

If the restriction $n\vert_\cQ$ is anisotropic, then since $n$ itself is isotropic, there are nonzero elements $q_1,q_2\in\cQ$ such that $n(q_1+q_2\cdot u)=0$, so that $n\bigl(q_1^{-1}\cdot(q_2\cdot u)\bigr)=-1$. Replacing $u$ by $q_1^{-1}\cdot(q_2\cdot u)$ we may assume $n(u)=-1$ (again, this will change the element $w$). Note then that the element $f=w^{\cdot 2}$ is an idempotent:
\[
w^{\cdot 2}*w^{\cdot 2}=\tau_e\left(\wb{w^{\cdot 2}}\right)\cdot\tau_e^2\left(\wb{w^{\cdot 2}}\right)
 =\wb{w^{\cdot 2}}\cdot\wb{w^{\cdot 2}}=w\cdot w=w^{\cdot 2},
\]
and 
\begin{gather*}
f*(1+u)=w\cdot(1-w^{\cdot 2}\cdot u)=w-u,\\
(1+u)*f=(1-w\cdot u)\cdot w=w-w\cdot u\cdot w=w-w^{\cdot 2}\cdot(w-u)=w-u.
\end{gather*}
Hence if $\tau_f$ denotes the corresponding order $3$ automorphism: $\tau_f(x)=f*(f*x)$, and we consider the Cayley algebra $(\cO,\diamond,n)$ with $x\diamond y=(f*x)*(y*f)$, the fixed subalgebra $\tilde\cQ=\{x\in\cO: f*x=x*f\}$ is a quaternion subalgebra of $(\cO,\diamond,n)$ with isotropic norm, because $1+u\in\tilde\cQ$ and $n(1+u)=1-1=0$. By the arguments above, $(\cO,*,n)$ is the split Okubo algebra.
\end{proof}

\begin{example}
According to Example \ref{ex:real}, there are two different conjugacy classes of order $3$ automorphisms $\tau$ of the split Cayley algebra $(\cC,\cdot,n)$ over $\RR$ such that $\cC_\tau$ is an Okubo algebra. By 
Theorem \ref{th:O_split_e}, $\cC_\tau$ is always the split Okubo algebra and hence, by Corollary \ref{co:idempotents_char_not3}, the real split Okubo algebra contains two conjugacy classes of idempotents.
\end{example}

%
%

\bigskip


\section{Order $3$ elements in $G_2$, $\charac\FF= 3$}\label{se:char3_autos}

Over a field $\FF$ of characteristic $3$, the situation for order $3$ elements in the automorphism group of a Cayley algebra is quite different. To begin with, only the split Cayley algebra admits order $3$ automorphisms (Theorem \ref{th:char3} below). In this case the group of automorphisms is the Chevalley group of type $G_2$, and given any split maximal torus and any root, all the elements $\exp(\delta)$, where  $\delta$ is a nonzero element in a root space (these elements make sense, as shown in Section \ref{se:Zorn_autos}), have order $3$ as its cube is $\exp(3\delta)=\exp(0)=\id$. 

\begin{remark}\label{re:Der_char3}
Over fields of characteristic $3$, the Lie algebra $\Der(\cC,\cdot,n)$ of derivations  of a Cayley algebra is not simple, as it contains the simple seven-dimensional ideal $\ad_\cC$ ($\ad_1=0$), which is a twisted form of $\frpsl_3$. Moreover, the quotient $\Der(\cC,\cdot,n)/\ad_{\cC}$ is isomorphic to $\ad_\cC$. (See, for example, \cite{Pabloetal} or \cite{AlonsoAlberto}.)
\end{remark}

The next definition  extends \cite[Definition 9.2]{CEKT} (see also \cite[Definition 22]{Eld_Strade}).

\begin{df}\label{df:idempotents_3}
Let $f$ be an idempotent of the Okubo algebra $(\cO,*,n)$  ($\charac\FF=3$). Then $f$ is said to be:
\begin{itemize}
\item \emph{quaternionic}, if its centralizer contains a para-quaternion algebra,
\item \emph{quadratic}, if its centralizer contains a para-quadratic algebra and no para-quaternion subalgebra,
\item \emph{singular}, otherwise.
\end{itemize}
\end{df}

It is proved in \cite[Proposition 9.9 and Theorem 9.13]{CEKT} that only the split Okubo algebra contains a quaternionic idempotent, and there is only one such idempotent.

The next result extends \cite[Lemma 21]{Eld_Strade}. The Peirce component $\cU$ in Remark \ref{re:canonical_basis} generates the split Cayley algebra. Hence any automorphism is determined by its action on $\cU$.

\begin{theorem}\label{th:char3}
Let $(\cC,\cdot,n)$ be a Cayley algebra over a field $\FF$ of characteristic $3$, and let $\tau$ be an order $3$ automorphism of $(\cC,\cdot,n)$. Then $(\cC,\cdot,n)$ is the split Cayley algebra and one of the following conditions holds:
\begin{enumerate}
\item $(\tau-\id)^2=0$. In this case, the Petersson algebra $\cC_\tau$ is the split Okubo algebra, $1$ is its (unique) quaternionic idempotent, and there exists a canonical basis of $\cC$ such that
\[
\tau(u_i)=u_i,\ i=1,2,\quad \tau(u_3)=u_3+u_2.
\]
In other words, $\tau=\exp(\delta)$, where $\delta=[L_{u_2},R_{v_3}]$, which is a derivation in the root space $\Der(\cC,\cdot,n)_{\veps_2-\veps_3}$. ($\veps_2-\veps_3$ is a long root.)

In particular, up to conjugacy, there is a unique such automorphism $\tau$. The Segre symbol of the nilpotent endomorphism $\tau-\id$ is $(2^2,1^4)$.

\item $(\tau-\id)^2\ne 0$ and there is a quadratic \'etale subalgebra $\cK$ of $\cC$ fixed elementwise by $\tau$. In this case, the Petersson algebra $\cC_\tau$ is an Okubo algebra, $1$ is a quadratic idempotent of $\cC_\tau$, and the Segre symbol of the nilpotent endomorphism $\tau-\id$ is $(3^2,1^2)$.

If $\FF$ is algebraically closed, then there is a canonical basis of $\cC$ such that 
\[
\tau(u_i)=u_{i+1}\quad \text{(indices modulo $3$)}
\]
so $\tau$ fixes $e_1$ and $e_2$, and it belongs to the normalizer  of the associated maximal torus. (Note that $\tau$ is the automorphism $\tau_{st}$ in \eqref{eq:tau_st}.)

\item There is a canonical basis such that
\[
\tau(u_i)=u_i,\ i=1,2,\quad \tau(u_3)=u_3+v_3-(e_1-e_2).
\]
In this case, $\tau(x)=w\cdot x\cdot w^{\cdot 2}$, with $w=1-v_3$. Hence (Theorem \ref{th:tau_w}) the Petersson algebra $\cC_\tau$ is a para-Cayley algebra with para-unit $w$. Also, $\tau=\exp(\delta)$, where $\delta=-\ad v_3:x\mapsto -[v_3,x]=x\cdot v_3-v_3\cdot x$ is a derivation in the root space $\Der(\cC,\cdot,n)_{-\veps_3}$. ($-\veps_3$ is a short root.)

In particular, up to conjugacy, there is a unique such automorphism $\tau$. The Segre symbol of the nilpotent endomorphism $\tau-\id$ is $(3,2^2,1)$.

\item There is a canonical basis such that
\[
\tau(u_i)=u_i,\ i=1,2,\quad \tau(u_3)=u_3+u_2+v_3-(e_1-e_2).
\]
In this case, the Petersson algebra $\cC_\tau$ is the split Okubo algebra, and $1$ is a singular idempotent of $\cC_\tau$. Such an automorphism $\tau$ is unique, up to conjugacy, and it is the composition of the (commuting) automorphisms in items \textup{(1)} and \textup{(3)}. The Segre symbol of the nilpotent endomorphism $\tau-\id$ is again $(3,2^2,1)$.
\end{enumerate}
\end{theorem}

\begin{proof}
$\Fix(\tau)^\perp$ is a subspace invariant under $\tau$ and, since $\tau-\id$ is nilpotent ($(\tau-\id)^3=0$), $\Fix(\tau)\cap\Fix(\tau)^\perp$ is a nonzero isotropic subspace. Hence $(\cC,\cdot,n)$ is the split Cayley algebra. 

As in \cite[Lemma 21]{EP96} either:
\begin{romanenumerate}
\item $(\tau-\id)^2=0$, in which case there is a quaternion subalgebra $\cQ$ of $\cC$ contained in $\Fix(\tau)$, or
\item $(\tau-\id)^2\ne 0$ and there is a quadratic \'etale subalgebra $\cK$ in $\Fix(\tau)$, or
\item $(\tau-\id)^2\ne 0$, $\Fix(\tau)=\FF 1\oplus\bigl(\Fix(\tau)\cap \cC_0\bigr)$, where $\cC_0=(\FF 1)^\perp$, and $\Fix(\tau)\cap\cC_0$ is a maximal isotropic subspace (hence of dimension $3$) of $\cC_0$.
\end{romanenumerate}

In case (i), for any nonisotropic $u\in\cQ^\perp$, $\cC=\cQ\oplus\cQ\cdot u$, and $\tau(u)=w\cdot u$, for a norm $1$ element $w\in\cQ$ such that $w^{\cdot 3}=1$. Then $w^{\cdot 2}=\bar w=w^{-1}$ and, as in the proof of \ref{th:tau_w}, we get $w^{\cdot 2}+w+1=0$. That is, $(w-1)^{\cdot 2}=0$ since the characteristic is $3$. We conclude that $\cQ$ is split: $\cQ\simeq M_2(\FF)$. But then the norm is universal in $\cQ$, and we may take $u\in\cQ^\perp$ with $n(u)=1$. Also, there is an isomorphism $\cQ\simeq M_2(\FF)$ that takes $w$ to $\left(\begin{smallmatrix} 1&1\\ 0&1\end{smallmatrix}\right)$. Then there is a canonical basis with $\cQ=\espan{e_1,e_2,u_1,v_1}$, $w=1+u_1$, and $u=u_2+v_2$. Therefore, we get
\[
\begin{split}
\tau(u_1)&=u_1,\quad \tau(v_1)=v_1,\\
 \tau(u_2)&=\tau(e_1\cdot u) =e_1\cdot\bigl((1+u_1)\cdot (u_2+v_2)\bigr)\\
     &=e_1\cdot(u_2+v_2-v_3)=u_2,\\
 \tau(v_2)&=\tau(e_2\cdot u)=e_2\cdot(u_2+v_2-v_3)=v_2-v_3,\\
 \tau(u_3)&=\tau(v_1\cdot v_2)=v_1\cdot (v_2-v_3)=u_3+u_2,\\
 \tau(v_3)&=\tau(u_1\cdot u_2)=v_3,
\end{split}
\]
Hence $\Fix(\tau)=\cQ + \{x\in\cQ : w\cdot x=x\}\cdot u=\espan{e_1,e_2,u_1,v_1,u_2,v_3}$ and we obtain the first possibility in the Theorem.

Case (ii) above corresponds to the second possibility in the Theorem, which will be developed in Theorem \ref{th:quadratic_auto}.

We are left with case (iii) above, so $\Fix(\tau)=\FF1\oplus\cW$, where $\cW$ is a three-dimensional isotropic subspace of $\cC_0$. Let $\cW'$ be another three-dimensional isotropic subspace of $\cC_0$ paired with $\cW$ (that is, the bilinear map $\cW\times\cW'\rightarrow \FF$, $(x,y)\mapsto n(x,y)$, is nondegenerate). Then $\cC_0=\cW\oplus\cW'\oplus\FF a$, with $\FF a=(\cW\oplus\cW')^\perp\cap\cC_0$ (nonisotropic). By Witt's Cancellation Theorem, the restriction of $n$ to $\FF 1\oplus\FF a$ is hyperbolic, so $-n(a)\in(\FF^\times)^2$, and we may scale $a$ so that $n(a)=-1$, that is, $a^{\cdot 2}=1$.

For any $x\in\cC_0$ and $y\in\cW$, since $\tau(y)=y$ we have
\[
n\bigl((\tau-\id)(x),y\bigr)=n\bigl(\tau(x),y\bigr)-n(x,y)=n\bigl(\tau(x),\tau(y)\bigr)-n(x,y)=0,
\]
so $(\tau-\id)(\cC_0)\subseteq \cW^\perp\cap\cC_0=\cW\oplus \FF a$. Then $0\ne (\tau-\id)^2(\cC_0)\subseteq \FF(\tau-\id)(a)$, and it follows that the Segre symbol of $(\tau-\id)\vert_{\cC_0}$ is $(3,2^2)$, so that the Segre symbol of $\tau-\id$ is $(3,2^2,1)$. Moreover, $(\tau-\id)(\cC_0)=\cW\oplus\FF a$.

Now, $\tau(a)=a+w$ for some $0\ne w\in \cW$. Also,
\[
n(a\cdot\cW,\cW)=n(a,\cW^{\cdot 2})\subseteq n(a,\cW)=0,\quad 
n(a\cdot\cW,a)=n(a^{\cdot 2},\cW)=n(1,\cW)=0,
\]
so we obtain
\[
a\cdot \cW=\cW\cdot a\subseteq \cW.
\]
Moreover, for any $x\in\cW$,
\[
w\cdot x=\tau(a)\cdot x-a\cdot x=\tau(a\cdot x)-a\cdot x=0
\]
because $a\cdot x$ is in $\cW\subseteq\Fix(\tau)$. Hence 
\[
w\cdot \cW=0=\cW\cdot w.
\]
The elements $e_1=\frac{1}{2}(1-a)$ and $e_2=\frac{1}{2}(1+a)$ are orthogonal idempotents, and let $\cC=\FF e_1\oplus\FF e_2\oplus\cU\oplus\cV$ be the corresponding Peirce decomposition: $\cU=e_1\cdot\cC\cdot e_2$, $\cV=e_2\cdot\cC\cdot e_1$. Then $\cW$ is invariant under multiplication by $e_1$ and $e_2$, so that $\cW=(\cW\cap\cU)\oplus(\cW\cap\cV)$. Besides, neither $\FF 1\oplus\cU$ nor $\FF 1\oplus\cV$ are subalgebras, so $\cW\cap\cU\neq\cW\neq\cW\cap\cV$. Changing $a$ by $-a$ if necessary (i.e., interchanging $e_1$ and $e_2$), we may (and will) assume $\dim_\FF(\cW\cap\cU)=2$, $\dim_\FF(\cW\cap\cV)=1$. Since $w\cdot\cW=0$, it follows that $\FF w=\cW\cap\cV$ and hence $a\cdot w=(e_1-e_2)\cdot w=-w$.

The subspace $\cW'$ may now be chosen so that $\cW'=(\cW'\cap\cU)\oplus(\cW'\cap\cV)$, with $\dim_\FF(\cW'\cap\cU)=1$ and $\dim_\FF(\cW'\cap\cV)=2$. Take the element $w'\in\cW'\cap\cU$ such that $n(w',w)=1$. Then $(\tau-\id)(w')\in\cW\oplus\FF a$, so
\[
\tau(w')=w'+\alpha a+u+\beta w
\]
for some $\alpha,\beta\in\FF$ and $u\in\cW\cap\cU$. From $\tau(a)=a+w$ we deduce:
\[
\begin{split}
\tau(e_1)&=\tau\left(\frac{1}{2}(1+a)\right)=\frac{1}{2}(1+a+w)=e_1-w,\\
\tau(e_2)&=\tau\left(\frac{1}{2}(1-a)\right)=\frac{1}{2}(1-a-w)=e_2+w.
\end{split}
\]
Then we get
\[
e_1-w=\tau(e_1)=-\tau(w'\cdot w)=-(w'+\alpha a+u+w)\cdot w=-w'\cdot w-\alpha a\cdot w=e_1+\alpha w,
\]
where we have used that $\cW\cdot w=0$. Hence $\alpha=-1$. Also,
\[
0=n(w')=n\bigl(\tau(w')\bigr)=n(a)+\beta n(w',w)=-1+\beta,
\]
so $\beta=1$ and, therefore,
\[
\tau(w')=w'-a+u+w.
\]
We are left with two cases:
\begin{itemize}
\item If $u=0$, we get
\[
(1-w)\cdot w'\cdot(1+w)=(w'+e_2)\cdot(1+w)=w'-e_1+e_2+w=\tau(w'),
\]
and for any $x\in\cW\cap \cU$, $(1-w)\cdot x\cdot(1+w)=x$ (recall that $\cW\cdot w=0=w\cdot \cW$). Since $\cU$ generates $\cC$, it follows that $\tau$ is the automorphism $x\mapsto (1-w)\cdot x\cdot(1-w)^{\cdot 2}$. (Note that the element $c=1-w$ satisfies $c^{\cdot 2}+c+1=0$.) 

Hence $\cC_\tau$ is para-Cayley (Theorem \ref{th:tau_w}). Let $u_3=w'$ and take $u_1,u_2\in\cW\cap\cU$ with $n(u_1\cdot u_2,u_3)=1$. In this situation, with $v_i=u_{i+1}\cdot u_{i+2}$ (indices modulo $3$), $\{e_1,e_2,u_1,u_2,u_3,v_1,v_2,v_3\}$ is a canonical basis of $(\cC,\cdot,n)$. Then $w=v_3$, because $n(u_i,w)=0$ for $i=1,2$ as $\cW$ is isotropic, and $n(u_3,w)=1$; and 
\[
\tau(u_i)=u_i,\ i=1,2,\quad \tau(u_3)=u_3-(e_1-e_2)+v_3,
\]
thus obtaining the third possibility of the Theorem.

\item If $u\ne 0$, let $u_3=w'$, $u_2=u$ and take $u_1\in\cW\cap\cU$ with $n(u_1\cdot u_2,u_3)=1$. Then $w=v_3$ again, and we obtain:
\[
\tau(u_i)=u_i,\ i=1,2,\quad \tau(u_3)=u_3-(e_1-e_2)+u_2+v_3.
\]
It remains to prove that the Petersson algebra $\cC_\tau=(\cC,*,n)$ is the split Okubo algebra. The element $e=1-v_3$ satisfies:
\begin{itemize}
\item $e*e=\tau(\bar e)\cdot\tau^2(\bar e)=(1+v_3)^{\cdot 2}=1-v_3=e$, so $e$ is an idempotent of $\cC_\tau$.

\item For any $x\in \cW$, $x*e=\tau(\bar x)\cdot\tau^2(\bar e)=-x\cdot (1+v_3)=-x$, as $\cW\cdot w=\cW\cdot v_3=0$, so $e*x=-e*(x*e)=-x$ too.

\item $a*e=\tau(\overline{e_1-e_2})\cdot\tau^2(\bar e)=-(e_1-e_2+v_3)\cdot (1-v_3)=-a-v_3+v_3=-a$, so that $e*a=-e*(a*e)=-a$ too.

\item $\tau(v_1)=\tau(u_2\cdot u_3)=u_2\cdot \tau(u_3)=v_1+u_2$, so $v_1*e=\tau(\bar v_1)\cdot\tau^2(\bar e)=-(v_1+u_2)\cdot(1-v_3)=-v_1-u_2+u_2=-v_1$, so that $e*v_1=-v_1$ too as before.
\end{itemize}

Hence $\espan{e_1,e_2,u_1,u_2,v_1,v_3}$ is contained in the centralizer in $\cC_\tau$ of $e$. But $\tau(v_2)=\tau(u_3\cdot u_1)=\tau(u_3)\cdot u_1=v_2-u_1-v_3$ and $v_2*e=\tau(\bar v_2)\cdot\tau^2(\bar e)=-(v_2-u_1-v_3)\cdot(1-v_3)=-v_2+u_1+u_1+v_3=-v_2-u_1+v_3$, and hence $e$ is not a para-unit of $\cC_\tau$. Consider the Cayley algebra $(\cC,\diamond,n)$ with $x\diamond y=(e*x)*(y*e)$, whose unity is $e$. Then $\tau':x\mapsto e*(e*x)$ is an automorphism of $(\cC,\diamond,n)$ with $(\tau')^3=\id$, and $\tau'\neq \id$ because $e$ is not a para-unit of $\cC_\tau=(\cC,*,n)$. As $\Fix(\tau')$, which is the centralizer in $\cC_\tau$ of $e$, has dimension at least $6$, $\tau'$ is necessarily an order $3$ element of $\Aut(\cC,\diamond,n)$ as in the first item of the Theorem, and hence $e$ is the quaternionic idempotent of the split Okubo algebra $\cC_\tau$. \qedhere
\end{itemize}
\end{proof}

\begin{remark}\label{re:char3_order3_generate_Aut}
Let $\FF$ be an algebraically closed field of characteristic $3$. In \cite[\S 5]{Ste68} it is shown that the automorphisms in items (1) and (3) of Theorem \ref{th:char3} generate the Chevalley group of type $G_2$, which is the whole group $\Aut(\cC,\cdot,n)$. In particular, $\Aut(\cC,\cdot,n)$ is generated (as an abstract group) by its order $3$ elements.
\end{remark}

\begin{corollary}\label{co:singular}
If an Okubo algebra $(\cO,*,n)$ contains a singular idempotent, then it is split.
\end{corollary}


\begin{example}\label{ex:tauKa}
Let $\cK$ be a quadratic \'etale algebra over a field $\FF$ of characteristic $3$, $a$ an element of $\cK$ of generic norm $1$, $\cW$ a free left $\cK$-module. Let $\{w_1,w_2,w_3\}$ be a $\cK$-basis of $\cW$ and let $\sigma:\cW\times\cW\rightarrow\cK$ be the nondegenerate hermitian form such that
\[
\sigma(w_i,w_i)=0,\quad \sigma(w_i,w_j)=-1\quad \text{for $1\leq i\neq j\leq 3$}.
\]
Finally, let $\Phi:\cW\times\cW\times\cW\rightarrow\cK$ be the alternating $\cK$-trilinear form such that $\Phi(w_1,w_2,w_3)=a$.

For any $x_1,x_2,x_3\in\cW$, $x_i=\sum_{j=1}^3a_{ij}\cdot w_j$, for some elements $a_{ij}\in\cK$. Then, since $\Phi(w_1,w_2,w_3)=a$ satisfies $n(a)=1$, and $\det\bigl(\sigma(w_i,w_j)\bigr)=1$ (because $\charac\FF=3$), we get
\begin{multline*}
n\bigl(\Phi(x_1,x_2,x_3)\bigr)=n\bigl(\det(a_{ij})\Phi(w_1,w_2,w_3)\bigr)\\
=\det(a_{ij})\overline{\det(a_{ij})}n\bigl(\Phi(w_1,w_2,w_3)\bigr)=n\bigl(\det(a_{ij})\bigr),
\end{multline*}
\[
\det\bigl(\sigma(x_i,x_j)\bigr)=\det\Bigl((a_{ij})\bigl(\sigma(w_i,w_j)\bigr)(\bar a_{ji})\Bigr)
  =\det(a_{ij})\det(\bar a_{ji})=n\bigl(\det(a_{ij})\bigr).
\]
Therefore, \eqref{eq:Phi_det} is satisfied and hence $\cC=\cK\oplus\cW$ is a Cayley algebra as in the paragraph previous to Lemma \ref{le:Stab_w}. Actually, it is the split Cayley algebra, as $n(w_i)=\sigma(w_i,w_i)=0$.

The $\cK$-linear map on $\cW$ permuting cyclically $w_1\mapsto w_2\mapsto w_3\mapsto w_1$ preserves $\sigma$ and $\Phi$, and hence we obtain an order $3$ automorphism $\tau$ of $\cC$ determined by
\[
\tau\vert_\cK=\id,\quad \tau(w_i)=w_{i+1}\ \text{(indices modulo $3$).}
\]
This automorphism will be denoted by $\tau_{\cK,a}$.\qed
\end{example}

Case (2) of Theorem \ref{th:char3} can be settled now.

\begin{theorem}\label{th:quadratic_auto}
Let $\tau$ be an order $3$ automorphism of a Cayley algebra $(\cC,\cdot,n)$ over a field $\FF$ of characteristic $3$ such that $(\tau-\id)^2\neq 0$ and $\tau$ fixes elementwise a quadratic \'etale subalgebra. Then the Petersson algebra $\cC_\tau$ is an Okubo algebra and $\tau$ is conjugate to an automorphism $\tau_{\cK,a}$ as in Example \ref{ex:tauKa}.

Moreover, two such automorphisms $\tau_{\cK,a}$ and $\tau_{\cK',a'}$ are conjugate if and only if there is an isomorphism $\varphi:\cK\rightarrow\cK'$ such that $\varphi(a)\in(\cK')^{\cdot 3}\cdot a'$ (i.e., there is an element $b'\in\cK'$, necessarily of norm $1$, such that $\varphi(a)=(b')^{\cdot 3}\cdot a'$).
\end{theorem}
\begin{proof} 
Let $\cK$ be a quadratic \'etale subalgebra whose elements are fixed by $\tau$. Let $\cW=\cK^\perp$ and consider $\sigma$, $\times$ and $\Phi$ as above. We will follow several steps.

\smallskip

\noindent\textbf{(i)}\quad There is an element $u\in\cW$ such that $\{u,\tau(u),\tau^2(u)\}$ is a $\cK$-basis of $\cW$:

This is trivial if $\cK$ is a field. Otherwise, $\cK=\FF e_1\oplus\FF e_2$, with $e_i^{\cdot 2}=e_i$, $i=1,2$, $e_1\cdot e_2=0=e_2\cdot e_1$. The Peirce subspaces $\cU=e_1\cdot \cC\cdot e_2$ and $\cV=e_2\cdot \cC\cdot e_1$ are invariant under $\tau$. Since $\cU$ and $\cV$ are isotropic subspaces coupled by the norm, and $\tau$ is an orthogonal transformation, the minimal polynomial of $\tau\vert_{\cU}$ and $\tau\vert_{\cV}$ is $T^3-1$. If $x\in\cU$ and $y\in\cV$ are elements such that $\{x,\tau(x),\tau^2(x)\}$ is a basis of $\cU$ and $\{y,\tau(y),\tau^2(y)\}$ is a basis of $\cV$, then we may take $u=x+y$. We also check with this argument that the Segre symbol of $\tau$ is $(3^2,1^2)$.

Moreover, if $\FF$ is algebraically closed, we may take a basis $\{u_1,u_2,u_3\}$ of $\cU$ with $n(u_1,u_2\cdot u_3)=1$ and $\tau(u_i)=u_{i+1}$ (indices modulo $3$). This shows that $\cC_\tau$ is the split Okubo algebra over $\FF$ (Definition \ref{df:Okubo}). If $\FF$ is not algebraically closed, extend scalars to the algebraic closure to show that $\cC_\tau$ is an Okubo algebra over $\FF$.

\smallskip

\noindent\textbf{(ii)}\quad With $\delta=\tau-\id$, for any $u\in\cW$ as in \textbf{(i)},
\[
0\ne n\bigl(\Phi(u,\tau(u),\tau^2(u)\bigr)=-n\bigl(\delta(u)\bigr)^3:
\]
Note first that, by linearity on each component, $\Phi\bigl(u,\tau(u),\tau^2(u)\bigr)=\Phi\bigl(u,\delta(u),\delta^2(u)\bigr)$. Now, for $x,y\in\cW$, as $\sigma$ is invariant under $\tau$,
\[
\begin{split}
\sigma\bigl(\delta(x),y\bigr)&=\sigma\bigl(\tau(x),y\bigr)-\sigma(x,y)\\
  &=\sigma\bigl(x,\tau^2(y)\bigr)-\sigma(x,y)=\sigma\bigl(x,(\delta^2-\delta)(y)\bigr).
  \end{split}
\]
Therefore, as $\delta^3=0$, for any $x,y\in\cW$,
\begin{equation}\label{eq:sigma_delta}
\sigma\bigl(\delta(x),\delta^2(y)\bigr)=0,\quad \sigma\bigl(\delta(x),\delta(y)\bigr)=-\sigma\bigl(x,\delta^2(y)\bigr).
\end{equation}
Thus, if $\alpha=n\bigl(\delta(u)\bigr)=\sigma\bigl(\delta(u),\delta(u)\bigr)$, we get $\sigma\bigl(u,\delta^2(u)\bigr)=-\alpha=\sigma\bigl(\delta^2(u),u\bigr)$, and 
\[
\begin{split}
0&\neq n\Bigl(\Phi\bigl(u,\tau(u),\tau^2(u)\bigr)\Bigr)
   =n\Bigl(\Phi\bigl(u,\delta(u),\delta^2(u)\bigr)\Bigr)\\
   &=\det\Bigl(\sigma\bigl(\delta^i(u),\delta^j(u)\bigr)\Bigr)_{0\leq i,j\leq 2}
    =\det\begin{pmatrix}*&*&-\alpha\\ *&\alpha&0\\ -\alpha&0&0\end{pmatrix}=-\alpha^3,
\end{split}
\]
as required.

\smallskip

\noindent\textbf{(iii)}\quad If $u$ and $v$ are two elements as in \textbf{(i)}, then 
\[
\Phi\bigl(v,\tau(v),\tau^2(v)\bigr)\in\cK^{\cdot 3}\cdot\Phi\bigl(u,\tau(u),\tau^2(u)\bigr):
\]

Indeed, given $u$ as in \textbf{(i)}, any other such $v$ is of the form $v=a\cdot u+\delta(w)$, for some $a\in\cK$ and $w\in\cW$. Then, since $\Phi$ is alternating,
\begin{multline*}
\Phi\bigl(v,\tau(v),\tau^2(v)\bigr)=\Phi\bigl(v,\delta(v),\delta^2(v)\bigr)\\
  =\Phi(a\cdot u,\delta(a\cdot u),\delta^2(a\cdot u)\bigr)=a^{\cdot 3}\cdot\Phi\bigl(u,\delta(u),\delta^2(u)\bigr).
  \end{multline*}
  
  \smallskip
  
\noindent\textbf{(iv)}\quad There exists an element $u$ as in \textbf{(i)} satisfying
\[
\sigma\bigl(\tau^i(u),\tau^i(u)\bigr)=0,\quad 
\sigma\bigl(\tau^i(u),\tau^j(u)\bigr)=-1,\quad\text{for $1\leq i\neq j\leq 3$:}
\]

This has been essentially done in \cite[p.~107]{EP96}. Here a slightly different argument will be used.

Let $u$ be an element as in \textbf{(i)} and let $a=\Phi\bigl(u,\tau(u),\tau^2(u)\bigr)$ and $\alpha=n\bigl(\delta(u)\bigr)$. By \textbf{(ii)}, $n(a)=-\alpha^3$. Take $\tilde u=(\alpha a^{-1})\cdot u$, then
\[
n\bigl(\delta(\tilde u)\bigr)=n(\alpha a^{-1})n\bigl(\delta(u)\bigr)=\frac{\alpha^2}{n(a)}\alpha=-1.
\]
Therefore we may assume that $n\bigl(\delta(u)\bigr)=-1$ and hence, by \eqref{eq:sigma_delta}, $\sigma\bigl(u,\delta^2(u)\bigr)=1$. For $a,b\in\cK$, consider the element $u'=u+a\cdot\delta(u)+b\cdot\delta^2(u)\in\cW$. Then $\delta(u')=\delta(u)+a\cdot \delta^2(u)$ so, using \eqref{eq:sigma_delta}, we have
\[
n\bigl(\delta(u')\bigr)=\sigma\bigl(\delta(u'),\delta(u')\bigr)
 =\sigma\bigl(\delta(u),\delta(u)\bigr)=n\bigl(\delta(u)\bigr)=-1.
\]
Hence $\sigma\bigl(u',\delta^2(u')\bigr)=1$. Now,
\[
\begin{split}
\sigma\bigl(u',\delta(u')\bigr)
 &=\sigma\bigl(u+a\cdot\delta(u)+b\cdot\delta^2(u),\delta(u)+a\cdot\delta^2(u)\bigr)\\
 &=\sigma\bigl(u,\delta(u)\bigr)+\sigma\bigl(u,\delta^2(u)\bigr)\cdot \bar a+a\cdot\sigma\bigl(\delta(u),\delta(u)\bigr)\\
 &=\sigma\bigl(u,\delta(u)\bigr)+(a-\bar a)\cdot \sigma\bigl(\delta(u),\delta(u)\bigr)\quad\text{(by \eqref{eq:sigma_delta})}\\
 &=\sigma\bigl(u,\delta(u)\bigr)+(a-\bar a)n\bigl(\delta(u)\bigr)=\sigma\bigl(u,\delta(u)\bigr)-(a-\bar a),
\end{split}
\]
and we may take $a$ so that $\sigma\bigl(u',\delta(u')\bigr)\in\FF$. But then we have
\[
\sigma\bigl(u',\delta(u')\bigr)=\sigma\bigl(\delta(u'),u'\bigr)=
 \sigma\bigl(u',(\delta^2-\delta)(u'\bigr)=1-\sigma\bigl(u',\delta(u')\bigr),
\]
so $\sigma\bigl(u',\delta(u')\bigr)=-1$. Finally,
\[
\begin{split}
\sigma(u',u')&=\sigma\bigl(u+a\cdot\delta(u),u+a\cdot\delta(u)\bigr)+
   \sigma\bigl(u,b\cdot\delta^2(u)\bigr)+\sigma\bigl(b\cdot\delta^2(u),u\bigr)\\
   &=n\bigl(u+a\cdot\delta(u)\bigr)+(b+\bar b),
\end{split}
\]
and we can take $b$ so that $\sigma(u',u')=0$. Therefore, the coordinate matrix of $\sigma$ in the $\cK$-basis $\{u',\delta(u'),\delta^2(u')\}$ is $\left(\begin{smallmatrix} 0&-1&1\\ -1&-1&0\\ 1&0&0\end{smallmatrix}\right)$, so that the coordinate matrix in the $\cK$-basis $\{u',\tau(u'),\tau^2(u')\}$ is $\left(\begin{smallmatrix} 0&-1&-1\\ -1&0&-1\\ -1&-1&0\end{smallmatrix}\right)$, as required. (For instance, $\sigma\bigl(u',\tau^2(u')\bigr)=\sigma\bigl(u',(\delta^2-\delta+\id)(u')\bigr) =1-(-1)+0=-1$.)

\smallskip

Now, if $u$ is an element as in \textbf{(iv)}, let $w_i=\tau^i(u)$ and we are in the situation of Example \ref{ex:tauKa}, with $a=\Phi\bigl(u,\tau(u),\tau^2(u)\bigr)$, so that $\tau$ is, up to conjugation, the automorphism $\tau_{\cK,a}$.

\medskip

If $\cK$ and $\cK'$ are quadratic \'etale subalgebras of $(\cC,\cdot,n)$, $a\in\cK$, $a'\in\cK'$ are norm $1$ elements and there is an isomorphism $\varphi:\cK\rightarrow\cK'$ and an element $b'\in\cK'$ such that $\varphi(a)=(b')^{\cdot 3}\cdot a'$, then $n(b')=1$. The automorphism $\tau=\tau_{\cK,a}$ is obtained by taking a $\cK$-basis $\{w_1,w_2,w_3\}$ of $\cK^\perp$ with $\sigma(w_i,w_i)=0$, $\sigma(w_i,w_j)=-1$ for $1\leq i\neq j\leq 3$, such that $\Phi(w_1,w_2,w_3)=a$ and $\tau(w_i)=w_{i+1}$ (indices modulo $3$). Similarly, denoting by $\sigma'$, $\Phi'$ the corresponding maps on $(\cK')^\perp$, there is a $\cK'$-basis $\{w_1',w_2',w_3'\}$ of $(\cK')^\perp$ with $\sigma'(w_i',w_i')=0$, $\sigma'(w_i',w_j')=-1$ for $1\leq i\neq j\leq 3$, such that $\Phi'(w_1',w_2',w_3')=a'$ and $\tau(w_i')=w_{i+1}'$ (indices modulo $3$). Then $\{b'\cdot w_1',b'\cdot w_2',b'\cdot w_3'\}$ is another $\cK'$-basis of $(\cK')^\perp$ and, since $n(b')=1$, it also satisfies $\sigma'(b'\cdot w_i',b'\cdot w_i')=0$, $\sigma'(b'\cdot w_i',b'\cdot w_j')=-1$, for $1\leq i\neq j\leq 3$, while $\Phi'(b'\cdot w_1',b'\cdot w_2',b'\cdot w_3')=(b')^{\cdot 3}\cdot a'=\varphi(a)$. Therefore, the isomorphism $\varphi:\cK\rightarrow\cK'$ extends to an automorphism of $(\cC,\cdot,n)$ by $\varphi(w_i)=b'\cdot w_i'$, $1\leq i\leq 3$, and $\varphi\circ\tau=\tau'\circ\varphi$.

Conversely, assume that the order $3$ automorphisms $\tau=\tau_{\cK,a}$ and $\tau'=\tau_{\cK',a'}$ are conjugate. Then there is an element $u\in\cW=\cK^\perp$ such that $\{u,\tau(u),\tau^2(u)\}$ is a $\cK$-basis of $\cW$ in which the coordinate matrix of $\sigma$ is  $\left(\begin{smallmatrix} 0&-1&-1\\ -1&0&-1\\ -1&-1&0\end{smallmatrix}\right)$ and $\Phi\bigl(u,\tau(u),\tau^2(u)\bigr)=a$ and similarly for $\tau'$. Let $\phi\in\Aut(\cC,\cdot,n)$ be such that $\phi\circ\tau'=\tau\circ\phi$. Then $\tilde\cK=\phi(\cK')$ is a quadratic \'etale subalgebra fixed elementwise by $\tau$: $\tilde\cK\subseteq \Fix(\tau)=\cK\oplus\delta^2(\cW)$. In $(\tilde\cK)^\perp$, the element $\tilde u=\phi(u')$ satisfies that $\{\tilde u,\tau(\tilde u),\tau^2(\tilde u)\}$ is a $\tilde\cK$-basis of $\tilde\cW=(\tilde\cK)^\perp$, and $\tilde\Phi\bigl(\tilde u,\tau(\tilde u),\tau^2(\tilde u)\bigr)=\tilde a=\phi(a')$. Then we must prove that there is an isomorphism $\varphi:\cK\rightarrow \tilde\cK$ such that $\varphi(a)\in(\tilde\cK)^{\cdot 3}\cdot\tilde a$, as then $\phi^{-1}\circ\varphi$ gives an isomorphism $\cK\rightarrow\cK'$ that takes $a$ to $(\cK')^{\cdot 3}\cdot a'$.

Let $0\ne k\in\cK$ orthogonal to $1$, so $\cK=\FF 1\oplus\FF k$. For any $x,y\in\cW$,
\[
\begin{split}
&n(x,y)=-n(x\cdot y,1)=n\bigl(\sigma(x,y),1\bigr),\\
&n(k\cdot x,y)=n(x\cdot y,k)=-n\bigl(\sigma(x,y)=-n\bigl(\sigma(x,y),k\bigr).
\end{split}
\]
Then,
\begin{equation}\label{eq:sigma_n_k}
\sigma(x,y)=-n(x,y)1+\frac{1}{n(k)}n(k\cdot x,y)k.
\end{equation}
The map $g:\cC\rightarrow\FF$ given by
\begin{equation}\label{eq:g}
 g(x)=n\bigl(x,\tau(\bar x)\cdot\tau^2(\bar y)\bigr)
\end{equation} 
is semilinear: $g(x+y)=g(x)+g(y)$, $g(\alpha x)=\alpha^3g(x)$, for any $x,y\in\cC$ and $\alpha\in\FF$ (see \cite[\S 5]{EP96}). Then, for any $x\in\cW$,
\[
\begin{split}
\Phi\bigl(x,\tau(x),\tau^2(x)\bigr)&=
  \sigma\bigl(x,\tau(x)\times \tau^2(x)\bigr)\\
   &=-n\bigl(x,\tau(x)\times\tau^2(x)\bigr)1+\frac{1}{n(k)}n\bigl(k\cdot x,\tau(x)\times\tau^2(x)\bigr)k\\
   &=-n\bigl(x,\tau(x)\cdot\tau^2(x)\bigr)1+\frac{1}{n(k)}n\bigl(k\cdot x,\tau(x)\cdot\tau^2(x)\bigr)k.
\end{split}
\]
But $(k\cdot y)\times(k\cdot z)=\bar k^2(y\times z)=-n(k)(y\times z)$, so we get:
\begin{equation}\label{eq:Phi_g}
\Phi\bigl(x,\tau(x),\tau^2(x)\bigr)=-g(x)1-\frac{1}{n(k)^2}g(k\cdot x)k.
\end{equation}

Now, $\tilde\cK$ is a quadratic \'etale subalgebra of $\Fix(\tau)=\cK\oplus\delta^2(\cW)=\cK\oplus\cK\cdot\delta^2(u)$. Since $n\bigl(\delta^2(\cW),\Fix(\tau)\bigr)=0$, $\tilde\cK=\FF 1\oplus\FF\tilde k$, with $\tilde k=k+l\cdot\delta^2(u)$ for some $l\in\cK$. Let $v=u+\frac{l+\bar l}{n(k)}k$. Then $n(v,1)=0$ and 
\[
n(\tilde k,v)=-(l+\bar l)+n\bigl(l\cdot\delta^2(u),u\bigr).
\]
But $\sigma\bigl(l\cdot\delta^2(u),u\bigr)=l\cdot\sigma\bigl(\delta^2(u),u\bigr)=l$ (recall that the coordinate matrix of $\sigma$ in the $\cK$-basis $\{u,\delta(u),\delta^2(u)\}$ is $\left(\begin{smallmatrix} 0&-1&1\\ -1&-1&0\\ 1&0&0\end{smallmatrix}\right))$, so
\[
n\bigl(l\cdot\delta^2(u),u\bigr)=\sigma\bigl(l\cdot\delta^2(u),u\bigr)+\sigma\bigl(u,l\cdot\delta^2(u)\bigr)=l+\bar l,
\]
and hence $n(\tilde k,v)=0$. Thus $v\in\tilde\cW=(\tilde\cK)^\perp$. By step \textbf{(iii)} above, $\tilde\Phi\bigl(v,\tau(v),\tau^2(v)\bigr)\in(\tilde\cK)^{\cdot 3}\cdot \tilde a$. As for \eqref{eq:Phi_g}, we have
\begin{equation}\label{eq:Phi_tilde_g}
\begin{split}
\tilde\Phi\bigl(v,\tau(v),\tau^2(v)\bigr)&=-g(v)1-\frac{1}{n(\tilde k)^2}g(\tilde k\cdot v)\tilde k\\
  &=-g(v)1-\frac{1}{n(k)^2}g(\tilde k\cdot v)\tilde k.
\end{split}
\end{equation}
On the other hand, $g(v)=g(u)$, because $g(k)=n(k,k^{\cdot 2})=0$, so 
\begin{equation}\label{eq:ktilde_v}
\begin{split}
\tilde k\cdot v&=\bigl(k+l\cdot\delta^2(v)\bigr)\cdot \Bigl(u+\frac{1}{n(k)}(l+\bar l)\cdot k\Bigr)\\
 &=k\cdot u-(l+\bar l)+\bigl(l\cdot\delta^2(u)\bigr)\cdot u+
   \bigl(l\cdot\delta^2(u)\bigr)\Bigl(\frac{1}{n(k)}(l+\bar l)\cdot k\Bigr).
\end{split}
\end{equation}
For any $x\in\cC$, $g(x)=g\bigl(\tau(x)\bigr)$, so $g\bigl(\delta(\cW)\bigr)=0$. In particular, the last term above is in $\delta^2(\cW)\subseteq\delta(\cW)$, so it is sent to $0$ by $g$. Also, $g(l+\bar l)=2g(l)=-g(l)$, and 
\[
\begin{split}
\bigl(l\cdot\delta^2(u)\bigr)\cdot u
 &=-\sigma\bigl(l\cdot\delta^2(u),u\bigr)+\bigl(l\cdot\delta^2(u)\bigr)\times u\\
 &=-l\cdot\sigma\bigl(\delta^2(u),u\bigr)+\bar l\cdot\bigl(\delta^2(u)\times u\bigr)\\
 &=-l+\bar l\cdot\bigl(\delta^2(u)\times u\bigr).
\end{split}
\]
As $\sigma\bigl(\delta^2(u),\delta^2(u)\times u\bigr)=\Phi\bigl(\delta^2(u),\delta^2(u),u\bigr)=0$, because $\Phi$ is alternating, 
\[
\delta^2(u)\times u\in\{x\in\cW : \sigma\bigl(\delta^2(u),x\bigr)=0\}=\delta(\cW),
\]
so $\bar l\cdot\bigl(\delta^2(u)\times u\bigr)\in\delta(\cW)$, $g\Bigl(\bar l\cdot\bigl(\delta^2(u)\times u\bigr)\Bigr)=0$, and $g\Bigl(\bigl(l\cdot\delta^2(u)\bigr)\cdot u\Bigr)=-g(l)$. 

Therefore, using \eqref{eq:ktilde_v}, $g(\tilde k\cdot v)=g(k\cdot u)-2g(l)-g(l)=g(k\cdot u)$, so that, by \eqref{eq:Phi_tilde_g}, we get
\[
\tilde\Phi\bigl(v,\tau(v),\tau^2(v)\bigr)=-g(u)-\frac{1}{n(k)}g(k\cdot u)\tilde k,
\]
and the isomorphism $\varphi:\cK\rightarrow \tilde\cK$ taking $k$ to $\tilde k$ satisfies that $\varphi(a)=\tilde\Phi\bigl(v,\tau(v),\tau^2(v)\bigr)$ belongs to $(\tilde\cK)^{\cdot 3}\cdot \tilde a$, as required.
\end{proof}

\begin{example}\label{ex:tau_Ka2}
Let $(\cC,\cdot,n)$ be the split Cayley algebra over a field $\FF$ of characteristic $3$, and let $\cK$ be a subalgebra isomorphic to $\FF\times\FF$, so $\cK=\FF e_1\oplus\FF e_2$, for orthogonal idempotents $e_1$ and $e_2$. For $0\ne \alpha,\beta\in\FF$, the elements $a=\alpha e_1+\alpha^{-1}e_2$ and $b=\beta e_1+\beta^{-1}e_2$ have norm $1$ and the automorphisms $\tau_{\cK,a}$ and $\tau_{\cK,b}$ are conjugate if and only if there is an automorphism $\varphi$ of $\cK$ such $\varphi(a)\in\cK^{\cdot 3}\cdot b$, if and only if there is a nonzero scalar $0\ne\mu\in\FF$ such that either $\alpha=\mu^3\beta$ or $\alpha=\mu^3\beta^{-1}$, if and only if $\FF^3\alpha\cup\FF^3\alpha^{-1}=\FF^3\beta\cup\FF^3\beta^{-1}$. If $\FF$ is perfect, we conclude that all these automorphisms are conjugate.
\end{example}

\begin{corollary}\label{co:Fperfect_tau_Ka}
Let $(\cC,\cdot,n)$ be the split Cayley algebra over a perfect field $\FF$ of characteristic $3$. Let $\cK$ and $\cK'$ be quadratic \'etale subalgebras, and let $a\in\cK$ and $a'\in\cK'$ be norm $1$ elements. Then the automorphisms $\tau_{\cK,a}$ and $\tau_{\cK',a'}$ are conjugate if and only if $\cK$ is isomorphic to $\cK'$. In particular, $\tau_{\cK,a}$ is conjugate to $\tau_{\cK,1}$ and, if $\FF$ is algebraically closed, all these automorphisms are conjugate.
\end{corollary}
\begin{proof} 
If $\varphi:\cK\rightarrow\cK'$ is an isomorphism, then either both $\cK$ and $\cK'$ are split, and we are in the situation of Example \ref{ex:tau_Ka2}, or $\cK$ and $\cK'$ are perfect fields, so $\varphi(a)\in\cK'=(\cK')^{\cdot 3}\cdot a'$.
\end{proof}

\bigskip


\section{Centralizers of order $3$ automorphisms, $\charac\FF=3$}\label{se:centralizers_3}

In this section the centralizers of the different order $3$ automorphisms in Theorem \ref{th:char3} will be computed.

For the automorphism in Theorem \ref{th:char3}.(1), this has been done in \cite[\S 10]{CEKT} as part of the computation of the algebraic group of automorphisms of the split Okubo algebra, which is not smooth. Here a different approach will be used, which sheds light on the corresponding results in \cite{CEKT}.

Actually, items (1) and (3) in Theorem \ref{th:char3} can be dealt with together. The notation of Section \ref{se:Zorn_autos} will be used throughout.

\begin{theorem}\label{th:parabolics}
Let $(\cC,\cdot,n)$ be the split Cayley algebra over a field $\FF$ of characteristic $3$.
\begin{enumerate}
\item For the order $3$ automorphism $\tau$ of $(\cC,\cdot,n)$ in item \textup{(1)} of Theorem \ref{th:char3}, $\CCentr_{\AAut(\cC,\cdot,n)}(\tau)$ is the derived subgroup of the parabolic subgroup $\bfP_{\{\veps_1\}}$.
\item  For the order $3$ automorphism $\tau$ of $(\cC,\cdot,n)$ in item \textup{(3)} of Theorem \ref{th:char3}, $\CCentr_{\AAut(\cC,\cdot,n)}(\tau)$ is the derived subgroup of the parabolic subgroup $\bfP_{\{\veps_2-\veps_1\}}$.
\end{enumerate}
\end{theorem}
\begin{proof}
For simplicity write $\bfG=\AAut(\cC,\cdot,n)$, $\cL=\Der(\cC,\cdot,n)$.

The automorphism in Theorem \ref{th:char3}.(1) is $\tau=\exp(\delta)$, where $\delta=[L_{u_2},R_{v_3}]$ is in the root space $\cL_{\veps_2-\veps_3}$. Note that $\delta^2=0$. More precisely,
\[
\begin{split}
&\delta(e_1)=\delta(e_2)=0,\\
&\delta(u_1)=0=\delta(u_2),\ \delta(u_3)=u_2\cdot(u_3\cdot v_3)-(u_2\cdot u_3)\cdot v_3=u_2,\\
&\delta(v_1)=0=\delta(v_3),\ \delta(v_2)=-v_3.
\end{split}
\]
Hence $\tau=\exp(\delta)=\id +\delta$. For $\varphi\in\bfG$, $\varphi\tau=\tau\varphi$ if and only if $\Ad_{\varphi}(\delta)=\delta$. Denote by $\bfH$ the centralizer of $\tau$ in $\bfG$, and by $\tilde\bfH$ the stabilizer in $\bfG$ of $\cL_{\veps_2-\veps_3}$, so that $\bfH$ is a subgroup of $\tilde\bfH$. Moreover, the derived subgroup of $\tilde\bfH$ is contained in $\bfH$, as its action on the one-dimensional space $\cL_{\veps_2-\veps_3}$ is trivial.

For any $\delta'\in\cL_{\alpha}$ with $\alpha\in\Phi^+\cup\{-\veps_1\}$, $\alpha+(\veps_2-\veps_3)$ is not a root, so $[\delta',\delta]=0$. Hence the $\bfU_\alpha$'s, with $\alpha\in\Phi^+\cup\{-\veps_1\}$ are contained in $\bfH$, and $\bfT$ and the $\bfU_{\alpha}$'s, with $\alpha\in\Phi^+\cup\{-\veps_1\}$, are contained in $\tilde\bfH$. It follows that $\tilde\bfH$ is the parabolic subgroup $\bfP_{\{\veps_1\}}$, and  $\tilde\bfH=\bfT\bfH$. But $\bfT\cap\bfH=\ker(\veps_2-\veps_3)$ is a maximal torus of the derived subgroup of $\bfP_{\{\veps_1\}}$, and hence $\bfH$ coincides with this derived subgroup.

\smallskip

As for the automorphism in Theorem \ref{th:char3}.(3), $\tau=\exp(\delta)$, where $\delta=-\ad_{v_3}$. For any $R$-point $\varphi\in\bfG(R)$, $\Ad_\varphi(\delta)=\delta$ occurs if and only if $\ad_{v_3}=\ad_{\varphi(v_3)}$ in $\cL\otimes_\FF R$, if and only if $\varphi(v_3)-v_3\in R1$. But $\varphi$ preserves the subspace orthogonal to $1$, so we get $\varphi(v_3)=v_3$. Denote again by $\bfH$ the centralizer of $\tau$ in $\bfG$, so $\bfH$ is the stabilizer of $v_3$, and by $\tilde\bfH$ the stabilizer in $\bfG$ of $\FF v_3$, so that $\bfH$ is a subgroup of $\tilde\bfH$. Moreover, the derived subgroup of $\tilde\bfH$ is contained in $\bfH$.

As before, for any $\delta'\in\cL_{\alpha}$ with $\alpha\in\Phi^+\cup\{\veps_1-\veps_2\}$, $\alpha+\veps_3$ is not a root, so $[\delta',\delta]=0$. Hence the $\bfU_\alpha$'s, with $\alpha\in\Phi^+\cup\{\veps_1-\veps_2\}$ are contained in $\bfH$, and $\bfT$ and the $\bfU_{\alpha}$'s, with $\alpha\in\Phi^+\cup\{\veps_1-\veps_2\}$, are contained in $\tilde\bfH$. It follows that $\tilde\bfH$ is the parabolic subgroup $\bfP_{\{\veps_2-\veps_1\}}$. As above, $\bfH$ coincides with its derived subgroup.
\end{proof}

\begin{remark}
Let $(\cC,\cdot,n)$ be the split Cayley algebra over a field $\FF$ of characteristic $3$, and let $\tau$ be its order $3$ automorphism in item \textup{(1)} of Theorem \ref{th:char3}. Then $e=1$ is the unique quaternionic idempotent of the split Okubo algebra $\cC_\tau=(\cC,*,n)$. By Proposition \ref{pr:fix_cent}, $\SStab_{\AAut(\cC,*,n)}(e)=\CCentr_{\AAut(\cC,\cdot,n)}(\tau)$, and this stabilizer is precisely $\AAut(\cC,*,n)_{\textrm{red}}$ (the largest smooth subgroup of $\AAut(\cC,*,n)$) by \cite[\S 10]{CEKT}. Theorem \ref{th:parabolics} gives a more clear picture of this last group. It turns out (see \cite[\S 11]{CEKT}) that the group of automorphisms $\AAut(\cC,*,n)$ factors as $\SStab_{\AAut(\cC,*,n)}(e)\boldmu_3^2$, where neither $\SStab_{\AAut(\cC,*,n)}(e)$ nor the non smooth subgroup $\boldmu_3^2$ are normal. 

The subgroup $\boldmu_3^2$ corresponds to a grading by $\ZZ_3^2$ of the split Okubo algebra (see \cite{EldGradingsSymmetric} or \cite[\S 4.6]{EK13}) and it is easy to check that it is not contained in the normalizer of any maximal torus. Actually, the maximal tori of $\AAut(\cC,*,n)$ have dimension one, because they are contained in $\AAut(\cC,*,n)_{\textrm{red}}$, which is the derived subgroup of the parabolic subgroup $\bfP_{\{\veps_1\}}$, and if $\boldmu_3$ is in the normalizer of a rank one torus, then it centralizes it. Hence, if $\boldmu_3^2$ were contained in the normalizer of a torus, it would centralize it. But $\boldmu_3^2$ is self-centralized, because the corresponding grading is fine with one-dimensional homogeneous components. 

The fact \cite[Theorem 3.15]{Platonov} that quasitori, and more generally supersolvable subgroups consisting of semisimple elements, are contained in normalizers of maximal tori over algebraically closed fields of characteristic $0$ has been an important tool in the classification of gradings on some exceptional Lie algebras (see, for instance, \cite{CandidoCristina}). The above shows that this is no longer true over fields of prime characteristic.
\end{remark}

The order $3$ automorphism $\tau$ in Theorem \ref{th:char3}.(4) is the composition of the automorphisms in items (1) and (3) of the same Theorem. Therefore, the unipotent radical $\bfU$ of the standard Borel subgroup $\bfB$, which is contained in the derived subgroups of both $\bfP_{\{\veps_1\}}$ and $\bfP_{\{\veps_2-\veps_1\}}$, centralizes $\tau$. Actually, it is the whole $\CCentr_{\AAut(\cC,\cdot,n)}(\tau)$.

\begin{theorem}\label{th:U}
Let $(\cC,\cdot,n)$ be the split Cayley algebra over a field $\FF$ of characteristic $3$. Then
the centralizer of the order $3$ automorphism of $(\cC,\cdot,n)$ in item \textup{(4)} of Theorem \ref{th:char3} is the unipotent radical $\bfU$ of the standard Borel subgroup $\bfB$.
\end{theorem}
\begin{proof}
Any $\varphi$ in $\CCentr_{\AAut(\cC,\cdot,n)}(\tau)$ stabilizes the subalgebra  of elements fixed by $\tau$: $\Fix(\tau)=\espan{1,u_1,u_2,v_3}$, so it stabilizes too its nilpotent radical: $\rad\bigl(\Fix(\tau)\bigr)=\espan{u_1,u_2,v_3}$, and its square $\FF v_3$. But the stabilizer of $\FF v_3$ has been shown to be the parabolic subgroup $\bfP_{\{\veps_2-\veps_1\}}$ in the proof of Theorem \ref{th:parabolics}. On the other hand, with the notations of Section \ref{se:Zorn_autos}, both the torus $\bfT$ and the subgroups $\bfU_{\pm(\veps_2-\veps_1)}$ are contained in the subgroup $\SSL(\cU)$, and hence so is the Levi subgroup $\bfL_{\{\veps_2-\veps_1\}}$. Hence $\CCentr_{\AAut(\cC,\cdot,n)}(\tau)=\bfU\bigl(\CCentr_{\AAut(\cC,\cdot,n)}(\tau)\cap \bfL_{\{\veps_2-\veps_1\}}\bigr)$. For any algebra $R\in\AlgF$, any $R$-point $\varphi$ of $\bfL_{\{\veps_2-\veps_1\}}$ has a matrix in our basis $\{u_1,u_2,u_3\}$ of the form
\[
\begin{pmatrix} a&c&0\\ b&d&0\\ 0&0&r\end{pmatrix}
\]
with $(ad-bc)r=1$. Hence $\varphi(v_3)=\varphi(u_1\cdot u_2)=(au_1+b u_2)\cdot (cu_1+du_2)=(ad-bc)v_3$, so
\[
\begin{split}
\varphi\tau(u_3)&=\varphi\bigl(u_3+u_2+v_3-(e_1-e_2)\bigr)=ru_3+cu_1+du_2+(ad-bc)v_3+(e_1-e_2),\\
\tau\varphi(u_3)&=e\tau(u_3)=ru_3+ru_2+rv_3+r(e_1-e_2).
\end{split}
\]
Hence, if $\varphi$ centralizes $\tau$, we get $r=1$, $c=0$, and $d=1$, so that $\varphi$ is an element in $\bfU_{\veps_2-\veps_1}\subseteq \bfU$. Therefore, $\CCentr_{\AAut(\cC,\cdot,n)}(\tau)=\bfU$.
\end{proof}

\smallskip

Finally, we must deal with the order $3$ automorphisms in Theorem \ref{th:char3}.(2), which correspond to the automorphisms in Example \ref{ex:tauKa}.  We will begin with the case that appears explicitly in Theorem \ref{th:char3}. Thus, consider the order $3$ automorphism $\tau_{st}$ of the split Cayley algebra $(\cC,\cdot,n)$ given in \eqref{eq:tau_st} that permutes cyclically the elements of our basis of the Peirce component $\cU$. That is,
\[
\tau(e_j)=e_j,\quad \tau(u_i)=u_{i+1},\quad \tau(v_i)=v_{i+1},
\]
for $j=1,2$, and for $i=1,2,3$ (indices taken modulo $3$).

In this case, the subalgebra of elements fixed by $\tau$ is $\Fix(\tau)=\espan{e_1,e_2,u,v}$, with $u=u_1+u_2+u_3$ and $v=v_1+v_2+v_3$. This is an associative algebra with nilpotent radical $\rad\bigl(\Fix(\tau)\bigr)=\FF u+\FF v$. Actually $(\FF u+\FF v)^{\cdot 2}=0$. Any element in the centralizer $\bfH=\CCentr_{\AAut(\cC,\cdot,n)}(\tau)$ stabilizes $\Fix(\tau)$ and its radical, and hence there is a natural homomorphism:
\begin{equation}\label{eq:Phi_tau_st}
\Phi:\bfH\longrightarrow \AAut\Bigl(\Fix(\tau)/\rad\bigl(\Fix(\tau)\bigr)\Bigr)\simeq\AAut(\FF\times\FF)\simeq \sfC_2,
\end{equation}
where $\sfC_2$ denotes the constant group scheme corresponding to the cyclic group of order $2$. The homomorphism $\Phi$ is surjective with a section given by the embedding of $\sfC_2$ into $\bfH$ corresponding to the order $2$ automorphism $\sigma\in\Aut(\cC,\cdot,n)$ given by $\sigma(e_1)=e_2$, $\sigma(u_i)=v_i$, $i=1,2,3$. Thus, $\bfH$ is the semidirect product 
\begin{equation}\label{eq:HkerPhi_C2}
\bfH=\ker\Phi\rtimes \sfC_2.
\end{equation}

As $u\cdot v=3e_1=0=v\cdot u$, the derivations $\ad_u$ and $\ad_v$ commute. Moreover, we have $\ad_u^3=\ad_v^3=0$. Their matrices in the canonical basis have integral entries, and if $T$ and $S$ are indeterminates,
\[
\exp(T\ad_{u})\exp(S\ad_v)=\exp(\ad_{Tu+Sv})\in\Aut(\cC\otimes_\FF\FF[T,S]),
\]
and hence, over any  algebra $R\in\AlgF$, it makes sense to consider the group $\bfK(R)=\{\exp(\ad_{tu+sv}):t,s\in R\}$, which is a subgroup of $\Aut(\cC\otimes_\FF R,\cdot,n)$. In this way we obtain a subgroup $\bfK$ of $\AAut(\cC,\cdot,n)$, which is contained in $\bfH$, and which is isomorphic to $\bfG_a^2$. Actually, $\bfK$ is contained in $\ker\Phi$ and its intersection with the subgroup $\SSL(\cU)$ is trivial (recall that $\SSL(\cU)$ is identified with the subgroup of $\AAut(\cC,\cdot,n)$ that fixes $e_1$ and $e_2$). Moreover, for any $R\in\AlgF$, any $R$-point $\varphi\in\ker(\Phi)$ satisfies $\varphi(e_1)=e_1+tu+sv$ for some $t,s\in R$, so $\varphi(e_1)=\exp(\ad_{-tu+sv})(e_1)$. Thus $\varphi\circ\exp(\ad_{tu-sv})$ fixes $e_1$ and $e_2=1-e_1$, and hence it belongs to $\bfH\cap\SSL(\cU)$. Therefore, we have
\begin{equation}\label{eq:kerPhi_KHSSL}
\ker\Phi=\bfK\rtimes \bigl(\bfH\cap\SSL(\cU)\bigr).
\end{equation}

There are two natural subgroups contained in $\bfH\cap\SSL(\cU)$. One of them is isomorphic to $\boldmu_3$, as $\boldmu_3$ embeds in $\bfH\cap\SSL(\cU)$ as follows:
\[
\begin{split}
\boldmu_3\ &\hookrightarrow\  \bfH,\\
 r&\mapsto \begin{cases} \id&\text{on $\FF e_1+\FF e_2$,}\\
                        r\id&\text{on $\cU$,}\\
                        r^2\id&\text{on $\cV$.}
                        \end{cases}
\end{split}
\]
The other natural subgroup of $\bfH\cap\SSL(\cU)$ is $\bfN\bydef \ker\Psi$, where $\Psi$ is the restriction homomorphism
\begin{equation}\label{eq:Psi}
\Psi:\bfH\longrightarrow \AAut\bigl(\Fix(\tau)\bigr).
\end{equation}
It is clear that the subgroups $\bfK$ and $\bfN$ commute. Given any $R\in\AlgF$ and any $\varphi\in \bfN(R)$, $\varphi(u_1)=au_1+bu_2+cu_3$ for $a,b,c\in R$. As $\varphi$ commutes with $\tau$, the matrix, relative to the basis $\{u_1,u_2,u_3\}$ of the restriction of $\varphi$ to $\cU$ is 
\begin{equation}\label{eq:matrix_abc}
\begin{pmatrix} a&c&b\\ b&a&c\\ c&b&a\end{pmatrix}
\end{equation}
with $a+b+c=1$ because $\varphi(u)=u$. Hence, in the basis $\{u_1,u_1-u_2,u_1+u_2+u_3\}$, the matrix is
\begin{equation}\label{eq:matrix_bc}
\begin{pmatrix} 1&0&0\\ c-b&1&0\\ c&c-b&1\end{pmatrix}
\end{equation}
and this shows that $\bfN$ is a two-dimensional unipotent group. Actually, the assignment
\[
(r,s)\mapsto\begin{pmatrix} 1&0&0\\ r&1&0\\ \frac{1}{2}(r+r^2)+s&\quad r\quad&\quad 1\quad\end{pmatrix}
\]
shows that $\bfN$ is isomorphic to $\bfG_a^2$.

\begin{proposition}\label{pr:centralizer_tau_st}
Let $(\cC,\cdot,n)$ be the split Cayley algebra over a field $\FF$ of characteristic $3$. Then the centralizer of the order $3$ automorphism $\tau_{st}$ in \eqref{eq:tau_st} is, with the notations above,
\[
\bfH=\Bigl(\bigl(\bfK\rtimes\boldmu_3\bigr)\times\bfN\Bigr)\rtimes \sfC_2.
\]
\end{proposition}
\begin{proof}
Because of \eqref{eq:HkerPhi_C2} and \eqref{eq:kerPhi_KHSSL}, it is enough to show that $\bfH\cap\SSL(\cU)=\boldmu_3\times \bfN$. But for any $R\in\AlgF$ and any $R$-point $\varphi\in\bfH\cap\SSL(\cU)$, if $\varphi(u_1)=au_1+bu_2+cu_3$ ($a,b,c\in R$), then the matrix of the restriction of $\varphi$ to $\cU$ is
the matrix in \eqref{eq:matrix_abc}, with determinant $1$, that is, with $(a+b+c)^3=1$, so $a+b+c\in\boldmu_3$. The result follows.
\end{proof}

Note that the centralizer in Proposition \ref{pr:centralizer_tau_st} is four-dimensional and \emph{not smooth}, because of the presence of the subgroup $\boldmu_3$.
The same situation appears in our next result, which deals with the general case in Theorem \ref{th:char3}.(2). We must consider the order $3$ automorphisms $\tau_{\cK,a}$ in Example \ref{ex:tauKa}:

\begin{theorem}\label{th:tauKa}
Let $(\cC,\cdot,n)$ be the split Cayley algebra over a field $\FF$ of characteristic $3$. Let $\tau$ be the order $3$ automorphism in Example \ref{ex:tauKa} and let $\bfH=\CCentr_{\AAut(\cC,\cdot,n)}(\tau)$ be its centralizer. Then, with the notations in this Example, there is a short exact sequence
\[
1\rightarrow \bigl(\bfK\rtimes\boldmu_{3[\cK]}\bigr)\times\bfN\rightarrow \bfH\rightarrow \sfC_2\rightarrow 1,
\]
where 
\begin{itemize}
\item 
$\boldmu_{3[\cK]}$ is the twisted form of $\boldmu_3$ such that for any $R\in\AlgF$, $\boldmu_{3[\cK]}(R)=\{ r\in\cK\otimes_\FF R: r^3=1,\ n(r)=1\}$ (see, for instance, \cite[p.~418]{KMRT}). Here $n$ denotes the $R$-extension of the generic norm of $\cK$. 

$\boldmu_{3[\cK]}$ embeds in $\bfH$ by sending any $r\in\boldmu_{3[\cK]}(R)$ to the automorphism of $\cC\otimes_\FF R$ that is the identity on $\cK\otimes_\FF R$ and $r$ times the identity on $\cW\otimes_\FF R$.
\item 
$\bfK$ is the abelian subgroup such that $\bfK(R)=\{\exp(\ad_x) : x\in \rad\bigl(\Fix(\tau)\bigr)\otimes_\FF R\}$ for any $R\in\AlgF$, so that $\bfK$ is isomorphic to $\bfG_a^2$.
\item $\bfN$ is the two-dimensional unipotent subgroup obtained as the kernel of the restriction map $\bfH\rightarrow \AAut\bigl(\Fix(\tau)\bigr)$. It is isomorphic too to $\bfG_a^2$.
\end{itemize}
\end{theorem}
\begin{proof}
Here $\Fix(\tau)=\cK\oplus\cK\cdot(w_1+w_2+w_3)$ and its radical is $\cK\cdot (w_1+w_2+w_3)$. As for Proposition \ref{pr:centralizer_tau_st}, we have a homomorphism as in \eqref{eq:Phi_tau_st} 
\[
\Phi:\bfH\longrightarrow \AAut\Bigl(\Fix(\tau)/\rad\bigl(\Fix(\tau)\bigr)\Bigr)\simeq\AAut(\cK)\simeq \sfC_2,
\]
and  $\ker\Phi=\bfK\rtimes \bigl(\bfH\cap\SSU(\cW,\sigma)\bigr)$, with $\bfK$ as above and where $\SSU(\cW,\sigma)$ is identified naturally with the subgroup of automorphisms of $(\cC,\cdot,n)$ which restrict to the identity on $\cK$. Moreover, $\Phi$ is surjective, as so it is over the algebraic closure $\overline{\FF}$, where we are in the situation of \eqref{eq:Phi_tau_st}. (Note that $\sfC_2$ is smooth.)

For $R\in\AlgF$, if $\varphi$ is an $R$-point of $\bfH\cap\SSU(\cW,\sigma)$, then $\varphi(w_1)=aw_1+bw_2+cw_3$ for $a,b,c\in\cK\otimes_\FF R$. As $\varphi$ commutes with $\tau$, the matrix of $\varphi$ in the basis $\{w_1,w_2,w_3\}$ is the matrix in \eqref{eq:matrix_abc}, whose determinant is $(a+b+c)^3$, which must be $1$. Write $r=a+b+c$. 

As $\varphi\in\SSU(\cW,\sigma)(R)$, $\sigma\bigl(\varphi(w_1),\varphi(w_1)\bigr)=\sigma(w_1,w_1)=0$, and this happens if and only if $-(a\bar b+a\bar c+b\bar a+b\bar c+c\bar a+a\bar b)=0$, if and only if $n(a,b)+n(a,c)+n(b,c)=0$, where $x\mapsto \bar x$ denotes the involution on $\cK\otimes_\FF R$ obtained by the extension of the nontrivial involution on $\cK$. Also $\sigma\bigl(\varphi(w_1),\varphi(w_2)\bigr)=\sigma(w_1,w_2)=-1$, which is equivalent to $-\bigl(n(a)+n(b)+n(c)+a\bar b+b\bar c+c\bar a\bigr)=-1$. This implies $a\bar b+b\bar c+c\bar a\in R$, so $a\bar b+b\bar c+c\bar a=\frac{1}{2}\bigl(n(a,b)+n(a,c)+n(b,c)\bigr)=0$. It follows then that $n(r)=n(a+b+c)=n(a)+n(b)+n(c)=1$. Therefore, $r\in\boldmu_{3[\cK]}$.

Now, with $\bfN$ being defined as the kernel of the restriction homomorphism $\Psi:\bfH\rightarrow \AAut\bigl(\Fix(\tau)\bigr)$, as in \eqref{eq:Psi}, the same arguments as for the proof of Proposition \ref{pr:centralizer_tau_st} give that $\bfN$ is a two-dimensional unipotent subgroup of $\bfH\cap\SSU(\cW,\sigma)$ and that $\ker\Phi=\bigl(\bfK\rtimes\boldmu_{3[\cK]}\bigr)\times\bfN$. More precisely, any $R$-point of $\bfN$ $\bigl(\subseteq \SSU(\cW,\sigma)\bigr)$ has a coordinate matrix as in \eqref{eq:matrix_abc}, relative to the basis $\{w_1,w_2,w_3\}$, with $a,b,c\in\cK\otimes_\FF R$ satisfying $a+b+c=1$, or as in \eqref{eq:matrix_bc}, relative to the basis $\{w_1,w_1-w_2,w_1+w_2+w_3\}$. But, as above, $a\bar b+b\bar c+c\bar a=0$, so $(1-b-c)\bar b+b\bar c+c(1-\bar b-\bar c)=0$. This gives $c+\bar b=(b-c)(\bar b-\bar c)=n(b-c)\in R$ and, as $b+\bar b\in R$, we get $r\bydef c-b\in R$. Hence $r^2=n(b-c)=c+\bar b$, so $c=r^2-\bar b=r+b$ and thus $c=-(r+r^2)-(b-\bar b)$. Let $\cK=\FF 1+\FF k$, with $n(1,k)=0$, then $\bar b-b=s k$ for some $s\in R$, and the assignment 
\[
(r,s)\mapsto \begin{pmatrix} 1&0&0\\ r&1&0\\ -(r+r^2)+sk&\quad r\quad&\quad 1\quad\end{pmatrix}
\]
gives an isomorphism $\bfG_a^2\simeq\bfN$.
\end{proof}

We will see in Remark \ref{re:dim_gO3} that the short exact sequence in Theorem \ref{th:tauKa} does not necessarily split, and even that the projection of $\bfH$ onto $\sfC_2$ may fail to be surjective for rational points.

\bigskip


\section{Idempotents in Okubo algebras, $\charac\FF =3$}\label{se:char3_idempotents}

The next definition is based on Theorem \ref{th:quadratic_auto}.

\begin{df}\label{df:cl}
Let $\FF$ be a field of characteristic $3$.
\begin{itemize}
\item
Let $\cK$ (respectively, $\cK'$) be a quadratic \'etale algebra over $\FF$, and let $a\in\cK$ (resp. $a'\in\cK'$) be a norm $1$ element. Then the pair $(\cK,a)$ is said to be \emph{equivalent} to $(\cK',a')$ if and only if there is an isomorphism $\varphi:\cK\rightarrow \cK'$ such that $\varphi(a)\in(\cK')^{\cdot 3}\cdot a'$. 

The equivalence class of the pair $(\cK,a)$ will be denoted by $[\cK,a]$.

\item Let $e$ be a quadratic idempotent of and Okubo algebra $(\cO,*,n)$ and let, as in \eqref{eq:e_tau}, $\tau_e:x\mapsto e*(e*x)$ be the associated order $3$ automorphism of $(\cO,*,n)$ and of the split Cayley algebra $(\cO,\cdot,n)$, where $x\cdot y=(e*x)*(y*e)$. Then $\tau_e$ is conjugate in $\Aut(\cO,\cdot,n)$ to an automorphism $\tau_{\cK,a}$ as in Example \ref{ex:tauKa}. Define the \emph{class of $e$} as $\cl(e)\bydef [\cK,a]$.
\end{itemize}
\end{df}

\begin{proposition}\label{pr:conjugation_class}
Let $(\cO,*,n)$ be an Okubo algebra over a field $\FF$ of characteristic $3$, and let $e$, $f$ be two quadratic idempotents. Then $e$ and $f$ are conjugate (by an element of $\Aut(\cO,*,n)$) if and only if $\cl(e)=\cl(f)$.
\end{proposition}
\begin{proof} 
Consider the order $3$ automorphisms $\tau_e$ and $\tau_f$ and the Cayley algebras $(\cO,\cdot,n)$ and $(\cO,\diamond,n)$ with multiplications given by $x\cdot y=(e*x)*(y*e)$ and $x\diamond y=(f*x)*(y*f)$ respectively. Since both $(\cO,\cdot,n)$ and $(\cO,\diamond,n)$ are the split Cayley algebra up to isomorphism, there is an isomorphism $\varphi:(\cO,\cdot,n)\rightarrow (\cO,\diamond,n)$. Note that $\varphi(e)=f$, because $e$ is the unity of $(\cO,\cdot,n)$ and $f$ the one in $(\cO,\diamond,n)$. Now $\tau_e$ is conjugate to some $\tau_{\cK,a}$ in $\Aut(\cO,\cdot,n)$ and $\tau_f$ to $\tau_{\cK',a'}$ in $\Aut(\cO,\diamond,n)$. Then $\varphi\circ\tau_e\circ\varphi^{-1}$ is conjugate to $\tau_{\varphi(\cK),\varphi(a)}$ and trivially $[\varphi(\cK),\varphi(a)]=[\cK,a]$.

Then $[\cK,a]=[\cK',a']$ if and only if $\varphi\circ\tau_e\circ\varphi^{-1}$ is conjugate to $\tau_f$ in $\Aut(\cO,\diamond,n)$, if and only if there is an automorphism $\psi\in\Aut(\cO,\diamond,n)$ such that $\psi\circ\varphi\circ\tau_e\circ(\psi\circ\varphi)^{-1}=\tau_f$, if and only if there is an isomorphism $\phi:(\cO,\cdot,n)\rightarrow (\cO,\diamond,n)$ such that $\phi\circ\tau_e=\tau_f\circ\phi$. In this case, for any $x,y\in\cO$,
\[
\begin{split}
\phi(x*y)&=\phi\Bigl(\tau_e\bigl(n(e,x)e-x\bigr)\cdot\tau_e^2\bigl(n(e,y)e-y\bigr)\Bigr)\\
 &=\tau_f\Bigl(n(f,\phi(x))f-\phi(x)\Bigr)\diamond\tau_f^2\Bigl(n(f,\phi(y))-\phi(y)\Bigr)\\
 &=\phi(x)*\phi(y),
\end{split}
\]
so $\phi\in\Aut(\cO,*,n)$, and $\phi(e)=f$. 

Conversely, if $\phi\in\Aut(\cO,*,n)$ and $\phi(e)=f$, then $\phi\circ\tau_e=\tau_f\circ\phi$, and $\phi$ is also an isomorphism $(\cO,\cdot,n)\rightarrow (\cO,\diamond,n)$.
\end{proof}

By \cite{Eld97}, if an Okubo algebra $(\cO,*,n)$ over a field $\FF$ of characteristic $3$ contains an idempotent, then either it is the split one (and this happens if and only if $g(\cO)=\FF^3$, where $g:\cO\rightarrow \FF$ is the semilinear map given by $g(x)=n(x,x*x)$ as in \eqref{eq:g}), or $g(\cO)$ is a cubic (purely inseparable) field extension of $\FF^3$, i.e., there is $\alpha\in\FF\setminus\FF^3$ such that $g(\cO)=\FF^3(\alpha)$. In the latter case, two such algebras are isomorphic if and only if $g(\cO_1)=g(\cO_2)$.

\begin{proposition}\label{pr:g_tau_Ka}
Let $(\cC,\cdot,n)$ be the split Cayley algebra over a field $\FF$ of characteristic $3$, let $\cK$ be a quadratic \'etale subalgebra, and $a\in\cK$ an element of norm $1$ as in Example \ref{ex:tauKa}. Consider the Petersson algebra $\cC_{\tau_{\cK,a}}$. Then 
\[
g\bigl(\cC_{\tau_{\cK,a}}\bigr)=\FF^3(a+\bar a)
\]
(the subfield of $\FF$ generated by $\FF^3$ and $a+\bar a$).
\end{proposition}
\begin{proof}
As in Example \ref{ex:tauKa}, $\cC=\cK\oplus\cK\cdot w_1\oplus\cK\cdot w_2\oplus\cK\cdot w_3$, with $\cW=\cK^\perp=\cK\cdot w_1\oplus\cK\cdot w_2\oplus\cK\cdot w_3$, $\sigma(w_i,w_i)=0$, $\sigma(w_i,w_j)-1$, for $1\leq i\neq j\leq 3$, and $\Phi(w_1,w_2,w_3)=a$. The automorphism $\tau=\tau_{\cK,a}$ restricts to the identity on $\cK$ and satisfies $\tau(w_i)=w_{i+1}$ (indices modulo $3$). Take $k\in\cK$ with  $n(1,k)=0\neq n(k)$, so that $\cK=\FF 1\oplus \FF k$.

Since $g(x)=n(x,x*x)$ is semilinear and $\tau$-invariant, $g\bigl(\cC_{\tau_{\cK,a}}\bigr)=\FF^3+g\bigl(\cK\cdot w_1\bigr)$, and it is enough to compute $g(w_1)$ and $g(k\cdot w_1)$. From
\[
\sigma(w_i,w_2\times w_3)=\Phi(w_i,w_2,w_3)=\begin{cases} 0&\text{if $i=2,3$,}\\ a&\text{if $i=1$,}
\end{cases}
\]
we obtain $w_2\times w_3=\bar a\cdot (w_2+w_3-w_1)$, so
\[
w_1*w_1=\tau(w_1)\cdot \tau^2(w_1)=w_2\cdot w_3=-\sigma(w_2,w_3)+w_2\times w_3=1+\bar a\cdot(w_2+w_3-w_1).
\]
Also, $n(\cK,\cW)=0$ and $n(x,y)=\sigma(x,y)+\sigma(y,x)=\sigma(x,y)+\overline{\sigma(x,y)}$ for any $x,y\in\cW$. Hence,
\[
\begin{split}
g(w_1)&=n\bigl(w_1,1+\bar a\cdot (w_2+w_3-w_1)\bigr)\\
 &=n(w_1,\bar a\cdot w_2)+n(w_1,\bar a\cdot w_3)\\
 &=\sigma(w_1,\bar a\cdot w_2)+\overline{\sigma(w_1,\bar a\cdot w_2)}+
   \sigma(w_1,\bar a\cdot w_3)+\overline{\sigma(w_1,\bar a\cdot w_3)}\\
   &=-a-\bar a-a-\bar a=a+\bar a,\\[6pt]
g(k\cdot w_1)&= n\bigl(k\cdot w_1),(k\cdot w_1)*(k\cdot w_1)\bigr)
 =n\bigl(k\cdot w_1,\tau(k\cdot w_1)\cdot\tau^2(k\cdot w_1)\bigr)\\
 &=n\bigl(k\cdot w_1,(k\cdot w_2)\times(k\cdot w_3)\bigr)=-n(k)n\bigl(k\cdot w_1,w_2\times w_3\bigr)\\
 &=-n(k)n\bigl(k\cdot w_1,\bar a\cdot(w_2+w_3-w_1)\bigr)\\
 &=-n(k)(k\cdot a+\overline{k\cdot a}).
\end{split}
\]
But $a=\alpha+\beta k$, $\alpha,\beta\in\FF$, with  $\alpha^2+\beta^2n(k)=1$, so $a+\bar a=-\alpha$, and
\[
-n(k)(k\cdot a+\overline{k\cdot a})=\beta n(k)^2=\begin{cases}
 0\ \text{if $\beta=0$, so that $\alpha=\pm 1\in\FF^3$,}&\\
 \frac{1}{\beta^3}\bigl(\beta^2n(k)\bigr)^2=\frac{1}{\beta^3}(1-\alpha^2)^2\in\FF^3(\alpha),&\text{otherwise.}
 \end{cases}
\]
Hence $g(k\cdot w_1)\in\FF^3(a+\bar a)$, as required.
\end{proof}

\begin{corollary}\label{co:g_tau_Ka}
The Petersson algebra $\cC_{\tau_{\cK,a}}$ is the split Okubo algebra if and only if $a\in\cK^{\cdot 3}$.
\end{corollary}
\begin{proof}
We have $g\bigl(\cC_{\tau_{\cK,a}}\bigr)=\FF^3$ if and only if $a+\bar a\in\FF^3$. But with $a=\alpha+\beta k$ as in the proof of Proposition \ref{pr:g_tau_Ka}, $a\in\cK^{\cdot 3}$ if and only if $\alpha=-(a+\bar a)\in\FF^3$. Indeed, if $a=(\mu+\nu k)^{\cdot 3}=\mu^3-n(k)\nu^3k$, then $-(a+\bar a)=\alpha=\mu^3$. And conversely, if $\alpha=\mu^3$, then $\beta^2n(k)=1-\alpha^2=1-\mu^6=(1-\mu^2)^3\in\FF^3$, and  either $\beta =0$ or $\beta k=-\frac{1}{\beta^2n(k)}(\beta k)^{\cdot 3}\in\cK^{\cdot 3}$. Hence both $\alpha$ and $\beta k$ are in $\cK^{\cdot 3}$, so $a\in\cK^{\cdot 3}$.
\end{proof}

Therefore, the situation for idempotents in Okubo algebras over fields of characteristic $3$ is now almost settled. Only a couple of things remain to be checked:

\begin{theorem}\label{th:Idempotents_Okubo_3}
Let $(\cO,*,n)$ be an Okubo algebra over a field $\FF$ of characteristic $3$.
\begin{enumerate}
\item If $\dim_{\FF^3}g(\cO)=8$, then $(\cO,*,n)$ contains no idempotents.
\item If $\dim_{\FF^3}g(\cO)=3$, $(\cO,*,n)$ contains a unique quadratic idempotent for each class $[\cK,a]$ with $a\not\in\cK^{\cdot 3}$ and $g(\cO)=\FF^3(a+\bar a)$.
\item If $(\cO,*,n)$ is the split Okubo algebra, then it contains:
\begin{enumerate}
\item a unique quaternionic idempotent,
\item a conjugacy class of quadratic idempotents for each isomorphism class of quadratic \'etale algebras,
\item a unique conjugacy class of singular idempotents.
\end{enumerate}
\end{enumerate}
\end{theorem}
\begin{proof} The uniqueness of the conjugacy class of singular idempotents in the split Okubo algebra needs to be proved, but if $e$ and $f$ are two such idempotents, using the notation in the proof of Proposition \ref{pr:conjugation_class}, $\tau_e$ is the unique, up to conjugacy, order $3$ automorphism of $(\cO,\cdot,n)$ in Theorem \ref{th:char3}.(4), and the same happens for $\tau_f$ in $\Aut(\cO,\diamond,n)$. Then there is an isomorphism $\varphi:(\cO,\cdot,n)\rightarrow (\cO,\diamond,n)$ with $\varphi\circ\tau_e=\tau_f\circ\varphi$ and, as in the proof of \ref{pr:conjugation_class}, $\varphi$ is an automorphism of $(\cO,*,n)$ and $\varphi(e)=f$.

Also, if $\dim_{\FF^3}g(\cO)=3$, only quadratic idempotents exist (see Corollary \ref{co:singular}), and there exists a conjugacy class of quadratic idempotentes for each class $[\cK,a]$ with $a\not\in\cK^{\cdot 3}$ and $g(\cO)=\FF^3(a+\bar a)$. It remains to show that these conjugacy classes are singletons. In other words, that the idempotents are fixed by automorphisms in this case. Note first that here $\FF$ is not perfect, and hence it is infinite. Let $e$ be a quadratic idempotent, let $(\cO,\cdot,n)$ be the Cayley algebra where $x\cdot y=(e*x)*(y*e)$ as in \eqref{eq:exye}, and let  $\bfG=\AAut(\cO,\cdot,n)$. The  automorphism group $\Aut(\cO,*,n)$ contains the group of rational points of  $\CCentr_{\bfG}(\tau)=\SStab_{\AAut(\cO,*,n)}(e)$. But the arguments in the proof of \cite[Theorem 13]{Eld99} show that the matrix group $\Aut(\cO,*,n)$ is a four-dimensional unipotent affine algebraic group over $\FF$ in the sense of \cite[Chapter 4]{Waterhouse}, and it is isomorphic, as an algebraic set, to the affine space of dimension $4$ (i.e.; its coordinate ring is isomorphic to the ring of polynomials in four variables over $\FF$). However, with the notation in Theorem \ref{th:tauKa}, $\CCentr_{\bfG}(\tau)(\FF)$ contains the closed subgroup $\bfK(\FF)\times\bfN(\FF)$, isomorphic to $\bfG_a^4(\FF)$. By dimension count, $\Aut(\cO,*,n)=\bfK(\FF)\times\bfN(\FF)$, and hence $e$ is fixed by any automorphism.
\end{proof}

\begin{remark}\label{re:dim_gO3}
In the case of $\dim_{\FF^3}g(\cO)=3$ in Theorem \ref{th:Idempotents_Okubo_3}, the arguments in the proof show that $\Aut(\cO,*,n)$ coincides with the group of rational points of $\bigl(\bfK\rtimes\boldmu_{3[\cK]}\bigr)\times\bfN$ in Theorem \ref{th:tauKa}, but then this gives
\[
\Aut(\cO,*,n)=\Bigl(\bigl(\bfK\rtimes\boldmu_{3[\cK]}\bigr)\times\bfN\Bigr)(\FF)=\bfH(\FF),
\]
and this shows that the projection $\bfH\rightarrow \sfC_2$ is not surjective on rational points. In particular, the short exact sequence in Theorem \ref{th:tauKa} does not split.
\end{remark}

Theorem \ref{th:Idempotents_Okubo_3} shows, in particular, that over an algebraically closed field of characteristic $3$, the Okubo algebra contains a unique quaternionic idempotent, a unique conjugacy class of quadratic idempotents, and a unique conjugacy class of singular idempotents. This was proved in \cite{Eld_Strade}.

\smallskip

The singular idempotents of the split Okubo algebra are strongly connected with the quaternionic idempotent.

\begin{proposition}\label{pr:quaternionic_singular}
Let $(\cO,*,n)$ be the split Okubo algebra over a field $\FF$ of characteristic $3$, and let $e$ be its quaternionic idempotent. Then:
\begin{romanenumerate}
\item The set of singular idempotents is given by:
\begin{multline*}
\{\text{singular idempotents of $(\cO,*,n)$}\}\\
=\{e+x: 0\neq x\in\Centr_{(\cO,*,n)}(e)\cap\Centr_{(\cO,*,n)}(e)^\perp\}.
\end{multline*}
\item The set of idempotents is:
\[
\{\text{idempotents of $(\cO,*,n)$}\}=\{e+x: x\in\Centr_{(\cO,*,n)}(e),\ x*x=0\}.
\]
\end{romanenumerate}
\end{proposition}
\begin{proof}
Let $\tau$ be the order $3$ automorphism in \eqref{eq:e_tau} attached to the quaternionic idempotent $e$, and let $(\cO,\cdot,n)$ be the associated Cayley algebra, with unity $1=e$, and with multiplication given by $x\cdot y=(e*x)*(y*e)$, so that $x*y=\tau(\bar x)\cdot\tau^2(\bar y)$, for any $x,y\in\cO$.

By Theorem \ref{th:char3}.(1), there is a canonical basis $\{e_1,e_2,u_1,u_2,u_3,v_1,v_2,v_3\}$ of $(\cO,\cdot,n)$ such that
\[
\tau(u_i)=u_i,\ i=1,2,\quad \tau(u_3)=u_3+u_2.
\]
The proof of Theorem \ref{th:char3} also shows that $\Fix(\tau)=\espan{e_1,e_2,u_1,u_2,v_1,v_3}$, so that 
$\cR\bydef \Centr_{(\cO,*,n)}(e)\cap\Centr_{(\cO,*,n)}(e)^\perp=\FF u_2+\FF v_3$. We have $\cR\cdot\cR=0$, so that $\cR*\cR=0$ too and hence, for any $x\in \cR$:
\[
(1+x)*(1+x)=\tau(\overline{1+x})\cdot\tau^2(\overline{1+x})=(1-x)\cdot (1-x)=1-2x=1+x,
\]
and $1+x=e+x$ is an idempotent. If $x\neq 0$, this idempotent is not the quaternionic idempotent, and hence to check that this is a singular idempotent, it is enough to find a three-dimensional isotropic subspace in its centralizer. But $x=\alpha u_2+\beta v_3$, with $\alpha,\beta\in\FF$. If $\alpha\neq 0=\beta$, then $\espan{u_2,v_3,v_1}$ is contained in the centralizer of $e+x$, if $\alpha=0\neq\beta$, then $\espan{u_2,v_3,u_1}$ works, and if $\alpha\neq 0\neq\beta$, then $\espan{u_2,v_3,(e_1-e_2)+\alpha^{-1}\beta u_1+\alpha\beta^{-1}v_1}$ does the job.

Conversely, if now $f$ is a singular idempotent, and $\tau'$ is the associated order $3$ automorphism: $\tau'(x)=f*(f*x)$ for any $x\in\cO$, the proof of Theorem \ref{th:char3} shows that there is a canonical basis $\{e_1,e_2,u_1,u_2,u_3,v_1,v_2,v_3\}$ of $(\cO,\diamond,n)$, where $x\diamond y=(f*x)*(y*f)$ for any $x,y$, such that
\[
\tau'(u_i)=u_i,\ i=1,2,\quad \tau'(u_3)=u_3-(e_1-e_2)+u_2+v_3,
\]
and it also shows that $e=1-v_3=f-v_3$ is the quaternionic idempotent of $(\cO,*,n)$, and that $\Centr_{(\cO,*,n)}(e)$ equals $\espan{e_1,e_2,u_1,u_2,v_1,v_3}$, so that $f=e+v_3$, and $v_3\in\Centr_{(\cO,*,n)}(e)\cap\Centr_{(\cO,*,n)}(e)^\perp$. This proves the first part.

Now, for any  $x\in \Centr_{(\cO,*,n)}(e)$, Proposition \ref{pr:fix_cent} shows that $\tau(x)=e*(e*x)=x$, and $e*x=x*e=n(e,x)e-x$. If $x*x=0$, then $0=x*x=\bar x^{\cdot 2}$, so that $x^{\cdot 2}=0$ and $n(x)=0=n(x,e)$. Hence $e*x=x*e=-x$ and $(e+x)*(e+x)=e+e*x+x*e=e-2x=e+x$.

Conversely, singular idempotents are of this form by the first part of the Proposition, so it is enough to show that quadratic idempotents are of this form too, and to do that we may extend scalars and assume that $\FF$ is algebraically closed. Let us prove first that any idempotent is contained in $\Centr_{(\cO,*,n)}(e)$. Since all quadratic idempotents are conjugate (by our assumption on $\FF$ and Theorem \ref{th:Idempotents_Okubo_3}), and since $\Centr_{(\cO,*,n)}(e)$ is invariant under automorphisms (because so is the unique quaternionic idempotent $e$), it is enough to show that $\Centr_{(\cO,*,n)}(e)$ contains a quadratic idempotent. With the notations in Theorem \ref{th:char3}.(1), the element $1+u_1$ is an idempotent and it is nonsingular as $u_1\not\in\Centr_{(\cO,*,n)}(e)^\perp$.

Now, if $0\ne f\in \Centr_{(\cO,*,n)}(e)$ is an idempotent, $f\in\Centr_{(\cO,*,n)}(e)$ so $f=a+b$ with $a\in\cQ=\espan{e_1,e_2,u_1,v_1}$ and $b\in\Centr_{(\cO,*,n)}(e)\cap\Centr_{(\cO,*,n)}(e)^\perp=\espan{u_2,v_3}$. Then $a*a=a$, so by the proof of Corollary \ref{co:idempotents_paraCayley}, $(a-1)^{\cdot 2}=0$, and $f=1+c=e+c$, with $c=(a-1)+b$. Moreover, $n(c)=0=n(1,c)$, so $c^{\cdot 2}=0$ and $c*c=\bar c^{\cdot 2}=0$, as required.
\end{proof}

\smallskip

Finally,  if a non split Okubo algebra $(\cO,*,n)$ over a field $\FF$ of characteristic $3$ contains idempotents, then $\dim_{\FF^3}g(\cO)=3$ and all the idempotents are quadratic. In this case any idempotent determines all the idempotents:

\begin{proposition}
Let $(\cO,*,n)$ be an Okubo algebra over a field $\FF$ of characteristic $3$ with $\dim_{\FF^3}g(\cO)=3$, and let $e$ be an idempotent. Then the set of idempotents is given by:
\[
\{\text{idempotents of $(\cO,*,n)$}\}=\left\{e+x: x\in\rad\bigl(\Centr_{(\cO,*,n)}(e)\bigr)\right\}.
\]
\end{proposition}
\begin{proof}
Let $(\cO,\cdot,n)$ be the Cayley algebra where $x\cdot y=(e*x)*(y*e)$ as in \eqref{eq:exye}, where $1=e$. The proof of  Theorem \ref{th:Idempotents_Okubo_3} shows that any idempotent is fixed by $\Aut(\cO,*,n)=\bfK(\FF)\times\bfN(\FF)$. By Theorems \ref{th:quadratic_auto} and \ref{th:tauKa}, $\cO=\cK\oplus\cK\cdot w_1\oplus\cK\cdot w_2\oplus\cK\cdot w_3$ as in Example \ref{ex:tauKa}, and the automorphism $\tau:x\mapsto e*(e*x)$ fixes elementwise the quadratic \'etale algebra $\cK$ and permutes cyclically the $w_i$'s. Then, with $w=w_1+w_2+w_3$ we have $\Centr_{(\cO,*,n)}(e)=\cK \oplus\cK\cdot w$, and
\[
\begin{split}
\Fix\bigl(\Aut(\cO,*,n)\bigr)&=\Fix\bigl(\bfK(\FF)\bigr)\cap\Fix\bigl(\bfN(\FF)\bigr)\\
 &=\Fix\bigl(\bfK(\FF)\bigr)\cap\bigl(\cK\oplus\cK\cdot w\bigr)\\
 &=\FF 1\oplus \cK\cdot w=\FF 1\oplus\rad\bigl(\Centr_{(\cO,*,n)}(e)\bigr).
\end{split}
\]
Any idempotent in $\FF 1\oplus\rad\bigl(\Centr_{(\cO,*,n)}(e)\bigr)$ projects into an idempotent in $\FF 1$, and hence it is of the form $e+x$, with $x\in \rad\bigl(\Centr_{(\cO,*,n)}(e)\bigr)$. Conversely, any element of the form $e+x$ with $x\in \rad\bigl(\Centr_{(\cO,*,n)}(e)\bigr)$ satisfies $(e+x)^{*2}=e^{*2}+e*x+x*e=e-2x=e+x$.
\end{proof}

\bigskip
%


\end{document}